\documentclass[a4paper,reqno,12pt]{amsart}
\usepackage{rotating,pdflscape,amssymb,hyperref,subfigure,psfrag,amsmath,eufrak,bbm}
\author[F. Benaych-Georges, A. Guionnet, M. Maida]{F. Benaych-Georges*, A. Guionnet$^\star$, M. Maida$^\sharp$.}

\title[Extreme eigenvalues of deformed random matrices]{Fluctuations of the extreme eigenvalues of finite rank deformations of  random matrices}

\date{\today}
\newcommand{\ulmu}{^{(\ell-1)}}
\newcommand{\ul}{^{(\ell)}}
\newcommand{\si}{\sigma}
\newcommand{\al}{\alpha}
\newcommand{\tta}{\theta}

\newcommand{\ovl}{\overline}
\newcommand{\udl}{\underline}
\newcommand{\bbm}{\begin{bmatrix}}
\newcommand{\ebm}{\end{bmatrix}}
\newcommand{\bes}{\begin{equation*}}
\newcommand{\ees}{\end{equation*}}
\newcommand{\be}{\begin{equation}}
\newcommand{\ee}{\end{equation}}
\newcommand{\beqy}{\begin{eqnarray}}
\newcommand{\eeqy}{\end{eqnarray}}
\newcommand{\beq}{\begin{eqnarray*}}
\newcommand{\eeq}{\end{eqnarray*}}

\newcommand{\GOE}{\operatorname{GOE}}
\newcommand{\GUE}{\operatorname{GUE}}

\newcommand{\sgn}{\operatorname{sgn}}
\newcommand{\lan}{\langle}
\newcommand{\ran}{\rangle}

\newcommand{\diag}{\operatorname{diag}}

\newcommand{\Pro}{\mathbb{P}}

\newcommand{\tr}{\operatorname{tr}}

\newcommand{\Tr}{\operatorname{Tr}}

\newcommand{\ninf}{\underset{n\to\infty}{\longrightarrow}}

\newcommand{\ssi}{if and only if }

\newcommand{\one}{\mathbbm{1}}

\newcommand{\E}{\mathbb{E}}

\newcommand{\R}{\mathbb{R}}
\newcommand{\C}{\mathbb{C}}

\newcommand{\z}{\mathbb{Z}}

\newcommand{\ud}{\mathrm{d}}

\newcommand{\pro}{probability }

\newcommand{\f}{\frac}
\newcommand{\ff}{\frac{1}}
\newcommand{\lf}{\left}
\newcommand{\ri}{\right}

\newcommand{\st}{such that }
\newcommand{\la}{\lambda}
\newcommand{\La}{\Lambda}

\newcommand{\ste}{\, ;\, }
\newcommand{\mc}{\mathcal }

\newcommand{\eps}{\varepsilon}

\def\e{{\epsilon}}
\def\ra{{\rightarrow}}

\newcommand{\wtX}{\widetilde{X_n}}

\newcommand{\wtl}{\widetilde{\la}}

\newcommand{\convas}{\overset{\textrm{a.s.}}{\longrightarrow}}

\newcommand{\bck}{\backslash}

\def\tr{{\rm Tr}}
\newtheorem{Th}{Theorem}[section]
\newtheorem{hyp}[Th]{Hypothesis}

\newtheorem{propo}[Th]{Proposition}

\newtheorem{assum}[Th]{Assumption}
\newtheorem{lem}[Th]{Lemma}

\newtheorem{rmq}[Th]{Remark}

\newenvironment{pr}{\noindent {\it Proof. }}{\hfill$\square$}

\long\def\symbolfootnote[#1]#2{\begingroup
\def\thefootnote{\fnsymbol{footnote}}\footnote[#1]{#2}\endgroup}
\begin{document}

\maketitle

\begin{center}
\small
* UPMC Univ Paris 6, LPMA,  Case courier 188, 4, Place Jussieu, 75252 Paris Cedex 05, France.\\
Email: florent.benaych@upmc.fr,\\
\mbox{}\\
$\star$ UMPA, ENS Lyon, 46 all\'ee d'Italie, 69364 Lyon Cedex 07,
France.\\
Email: aguionne@umpa.ens-lyon.fr,\\
\mbox{} \\
$\sharp$ Universit\'e Paris-Sud, Laboratoire de Math\'ematiques,
B\^atiment 425,
Facult\'e des Sciences,
91405 Orsay Cedex, France.\\
Email: mylene.maida@math.u-psud.fr.\\
\end{center}

\mbox{}\\

\normalsize

\begin{abstract}
 Consider a deterministic self-adjoint matrix $X_n$ with spectral measure converging to a compactly supported
probability measure, the largest and smallest eigenvalues converging to the edges of the limiting measure. We perturb this matrix by adding
  a random finite rank matrix with delocalised eigenvectors and study the extreme eigenvalues of the deformed model. We give necessary conditions on the deterministic matrix $X_n$ so that
the eigenvalues converging out of the bulk exhibit Gaussian fluctuations,    
whereas the eigenvalues sticking to the edges are very close to the eigenvalues of the non-perturbed model and fluctuate in the same scale.\\
We generalize  these results to the case when $X_n$ is random and get similar behavior when we deform
some classical models such as  Wigner or Wishart matrices with rather general entries or the so-called matrix models.
\end{abstract}

\mbox{}\\

\noindent {\sl Math. Subj. Class.:} 60B20, 60F05.\\
\mbox{}\\

\noindent
This work was   supported by the \emph{Agence Nationale de la
Recherche} grant ANR-08-BLAN-0311-03.

\noindent Submitted January 04, 2011,
 final version accepted July 23, 2011.

\newpage

\tableofcontents

\section{Introduction}\label{introduction}
Most of the  spectrum of 
a large matrix is not much  altered if one adds a finite 
rank perturbation to the matrix, simply because of Weyl's interlacement properties of the eigenvalues.
 But the extreme eigenvalues, depending on the strength of the perturbation, can either
stick to the extreme eigenvalues of the non-perturbed matrix or deviate to some
larger values. This phenomenon was made precise
in \cite{BBP}, where a sharp phase transition, known as the {\it BBP transition}
\cite{gen1,gen2,gen3,gen4},
 was exhibited for finite rank perturbations of a complex Gaussian  Wishart matrix. In this case,
 it was shown that if the strength of the perturbation is above a threshold,
the largest eigenvalue of the perturbed matrix 
deviates away from the bulk and has then
Gaussian fluctuations, otherwise it sticks to the bulk and fluctuates according to the
Tracy-Widom law. The fluctuations of the extreme eigenvalues
which deviate from the bulk were studied as well
when the non-perturbed matrix is a Wishart (or  Wigner)
matrix with non-Gaussian entries;
they were shown to be Gaussian if the perturbation is chosen randomly
with i.i.d. entries  in \cite{bai-yao-TCL},
or with completely delocalised eigenvectors \cite{FP07,FeralPeche09},
 whereas in  \cite{CDF09}, a non-Gaussian behaviour was
exhibited when the perturbation has localised eigenvectors. The influence 
of the localisation of the eigenvectors of the perturbation
was studied more precisely in \cite{CDF09b}.\\

 In this
paper, we also focus  on the behaviour of the extreme eigenvalues
 of a  finite rank
perturbation of a large matrix, this time in the framework where the large matrix is
 deterministic whereas the perturbation has delocalised 
random eigenvectors. 
 We show that the eigenvalues which deviate
away from the bulk have Gaussian fluctuations, whereas those
which stick to the bulk are extremely close to the extreme eigenvalues
of the non-perturbed matrix.  In a one-dimensional perturbation situation, we
can as well study the fluctuations of the next eigenvalues, for instance showing that if the first eigenvalue deviates from the bulk, the second
eigenvalue will stick to the first eigenvalue of the non-perturbed matrix,
whereas if the first eigenvalue sticks to the bulk, the second eigenvalue 
will be very close to the second eigenvalue of the non-perturbed matrix.
Hence, for a one dimensional 
perturbation, the eigenvalues which stick to the
bulk will fluctuate as the eigenvalues of the non-perturbed matrix.
We can also extend these results
beyond the case when the non-perturbed matrix is deterministic. 
In particular, if the non-perturbed matrix is a
   Wishart (or Wigner) matrix with rather general
entries, or a matrix model, we can use the universality
of the fluctuations of the extreme
eigenvalues of these random matrices
to show   that the $p$th extreme eigenvalue which sticks to the 
bulk fluctuates according to the $p$th dimensional Tracy-Widom law.  This proves the universality
of the BBP transition at the fluctuation level, provided the
perturbation is delocalised and random. \\
The reader should notice however that we do not deal with the asymptotics of eigenvalues corresponding to critical deformations.
This probably requires a case-by-case analysis and may depend on the model under  consideration.\\

Let us now describe more precisely the models we will be dealing with. We consider a deterministic  self-adjoint 
 matrix $X_n$ with eigenvalues   $\la_1^{n}\le \cdots\le \la_n^{n}$
satisfying the following hypothesis.
\begin{hyp} \label{hypspec}
The spectral measure $\mu_n:=n^{-1}\sum_{l=1}^n \delta_{\lambda_l^n}$ of $X_n$ converges towards a deterministic
  \pro measure $\mu_X$ with compact support. Moreover, 
 the smallest and largest eigenvalues of $X_n$  converge respectively   
 to $a$ and $b$, the lower and upper bounds
of the support of $\mu_X$.
\end{hyp}

We  study the eigenvalues  $\wtl_1^n \le \cdots\le \wtl_n^n$ of a perturbation $\wtX:=X_n+R_n$ obtained from $X_n$
by adding a finite rank matrix $R_n=\sum_{i=1}^r \theta_i u_i^n u_i^{n^*}.$
We shall assume $r$ and the 
 $\theta_i$'s to be   deterministic and independent of $n$, but the
column vectors $(u_i^n)_{1\le i\le r}$ chosen randomly
as follows. Let $\nu$ be a probability measure on $\R$ or $\C$ satisfying
\begin{assum} \label{hyponG}
 The probability measure $\nu$ 
 satisfies a logarithmic Sobolev inequality, is centred and has variance one. If  $\nu$ is not supported on $\R$,
we assume moreover that
its real part and its imaginary part are independent and identically distributed. 
\end{assum}
We consider now a random vector
 $v^n=\frac{1}{\sqrt{n}} (x_1,\ldots,x_n)^T$ 
with $(x_i)_{1\le i\le n}$ i.i.d. real or complex 
 random variables with law $\nu$. Then
\begin{enumerate}
\item Either the $u_i^n$'s ($i=1, \ldots, r$) are independent  copies of $v^n$
\item Or $(u_i^n)_{1\le i\le r}$ are obtained
by the Gram-Schmidt orthonormalisation 
of $r$ independent copies of a vector $v^n.$ 
\end{enumerate}

We shall refer to the model (1) as   the {\it i.i.d. model} and to the model (2) as the {\it orthonormalised model}. \\

Before giving a rough statement of our results, let us make a few remarks.\\
We first recall that a probability measure $\nu$ is said to satisfy a {\it logarithmic Sobolev inequality} with constant $c$
if, for any differentiable funtion $f$ in $L^2(\nu),$
$$ \int f^2 \log \frac{f^2}{\int f^2 d\nu} d\nu \le 2c \int |f^\prime|^2 d\nu.$$
It is well known that  a logarithmic Sobolev inequality implies sub-gaussian tails and concentration estimates.
The concentration properties of the measure $\nu$ that will be useful in the proofs  are detailed in Section \ref{section.concentration.estimates}
of the Appendix.\\
 In the orthonormalised model, if 
$\nu$ is the standard real (resp. complex) Gaussian law, $(u_i^n)_{1\le i\le r}$ follows the
uniform law on the set of orthogonal 
 random vectors  on the unit sphere of $\R^n$ (resp. $\C^n$) and by invariance by conjugation,
the model coincides with the one studied in \cite{benaych-rao.09}.\\
For a general probability measure
$\nu$, the $r$ i.i.d.  
random vectors obtained are not necessarily linearly independent almost surely, so
that the orthonormal vectors described in (2) are not always almost surely
well defined. However, as the dimension goes to infinity, they are
well defined with overwhelming probability when $\nu$ satisfies Assumption \ref{hyponG} .
This means the following: we shall say that a sequence of events $(E_n)_{n \ge 1}$ occurs with {\it overwhelming  probability}\footnote{Note that this is a bit different from what is called {\it overwhelming probability}
by Tao and Vu but will be sufficient for our purpose.}
if there exist two constants  $C,\eta>0$ independent of 
$n$ such that  
  $$ \Pro(E_n) \ge 1-Ce^{-n^\eta}.$$  
Consequently, in the sequel,
we shall restrict ourselves to the event when the model (2) is well defined without mentioning it explicitly.\\

In this work, we  study the asymptotics 
of the  eigenvalues
of $\wtX$ outside  the spectrum of $X_n$.

It has already been observed in similar situations, see \cite{BBP},  that these eigenvalues converge
to the boundary of the support of $X_n$ if the $\theta_i$'s are small enough,
whereas for sufficiently large values of the $\theta_i$'s, 
they stay away from the bulk of $X_n$. 
More precisely, if we let $G_{\mu_X}$ be the Cauchy-Stieltjes transform of $\mu_X$, defined, for  $z<a$ or $z>b,$ 
by the formula
$$G_{\mu_X}(z)=\int\frac{1}{z-x}\ud\mu_X(x) ,$$
then the eigenvalues of $\wtX$ outside  the bulk  converge to the solutions of $ G_{\mu_X}(z)=\theta_i^{-1}$ if they exist.

Indeed, if we let $$\overline\theta:=\ff{\lim_{z\downarrow b}G_{\mu_X}(z)}\ge 0,\qquad \underline\theta:=\ff{\lim_{z\uparrow a}G_{\mu_X}(z)}\le 0$$ and
 $$\rho_{\theta}:=\begin{cases}G_{\mu_X}^{-1}(1/\theta)&\textrm{if }\theta
\in (-\infty,\underline\theta)\cup (\overline\theta,+\infty),\\
a&\textrm{if  }\theta\in [\underline\theta,0),\\
b&\textrm{if   }{\theta}\in (0,\overline\theta],
 \end{cases}
$$
then
we have the following theorem. Let $r_0\in \{0, \ldots, r\}$
 be \st $$\theta_{1}\le \cdots \le\theta_{r_0}<0<\theta_{r_0+1}\le
\cdots\le \theta_r.$$

\begin{Th}\label{241109.17h50cor11} Assume that Hypothesis \ref{hypspec} and Assumption \ref{hyponG} are satisfied.  For all $i\in \{1, \ldots, r_0\}$, we have
 $$\wtl_i^{n}\convas \rho_{\theta_i}$$
and for all $i\in\{{r_0+1}, \ldots, r\}$, $$ \wtl_{n-r+i}^{n}\convas \rho_{\theta_{i}}.$$
Moreover, for all $i>r_0$ (resp. for all $i\ge r-r_0$) independent of $n$,
 $$\wtl_i^n\convas a\qquad\textrm{(resp. $\wtl_{n-i}^n\convas b$).}$$\end{Th}

The uniform case was proved in
 \cite[Theorem 2.1]{benaych-rao.09} and  we will follow a similar strategy to prove 
Theorem \ref{241109.17h50cor11} under our assumptions in Section \ref{sec.ps}.\\

The main object of this paper is to study the fluctuations
of the extreme eigenvalues of $\wtX.$ Precise statements will be given in Theorems \ref{theogauss}, \ref{942010.1}, \ref{theostick}, \ref{exact1d}
 and   \ref{corstick}. For any $x$ such that $x \le a$ or $x \ge b,$ we denote by  $I_x$ the set of indices $i$ such that 
$\rho_{\theta_i} = x.$ The results roughly state as follows. 
\begin{Th}\label{mainint}
Under  additional  hypotheses,
\begin{enumerate}
\item 
 Let  $\alpha_1<\cdots<\alpha_q$ be the different values
of the $\theta_i$'s \st  $\rho_{\theta_i}\notin\{a,b\}$ and denote, for each $j$, 
 $k_j=|I_{\rho_{\alpha_j}}|$ and $q_0$ the largest index so that $\alpha_{q_0}< 0$.
Then, the law
of the random vector $$\left(\sqrt{n}(\wtl_{i}^n-\rho_{\alpha_j}), i\in I_{\rho_{\alpha_j}}
\right)_{
1\le j\le q_0}\cup \left(\sqrt{n}(\wtl_{n-r+i}^n-\rho_{\alpha_j}),
 i\in I_{\rho_{\alpha_j}} \right)_{q_0+1\le j\le q}$$   converges to the law of the eigenvalues 
of $(c_{\alpha_j} M_{j})_{1\le j\le q}$ with the $M_{j}$'s being independent
matrices  following the 
law of a $k_j\times k_j$ matrix
from the GUE or the GOE, depending whether
$\nu$ is supported on the complex plane or the real line.
The constant $c_{\alpha_j}$ is explicitly defined  in Equation \eqref{28110.1}.\\

\item If none of the $\theta_i$'s are critical (i.e. equal to $\udl{\tta}$ or $\ovl{\tta}$), with overwhelming  probability,
the extreme eigenvalues  converging to $a$ or $b$ are at
distance at most $n^{-1+\e}$ of the extreme eigenvalues of
$X_n$ for some $\e>0$.\\

\item If $r=1$ and $\theta_1 = \theta >0$, we have the following more precise picture
about the extreme eigenvalues:\\
\begin{itemize}
 \item If $\rho_\theta>b$, $\sqrt{n}(\widetilde\lambda^n_n-\rho_\theta)$
converges towards a Gaussian variable, whereas $n^{1-\e}(\widetilde\lambda^n_{n-i}
-\lambda_{n-i+1})$ vanishes  in probability as $n$ goes to infinity for any fixed $i\ge 1$
and some $\e>0$. \\
\item If $\rho_\theta=b$ and $\theta \neq \overline \theta,$  $n^{1-\e}(\widetilde\lambda^n_{n-i}
-\lambda_{n-i})$  vanishes  in probability as $n$ goes to infinity for any fixed $i\ge 1$
and some $\e>0$.\\
\item For any fixed $j\ge 1$, $n^{1-\e}(\widetilde\lambda^n_{j}
-\lambda_{j})$  vanishes  in probability as $n$ goes to infinity for  some $\e>0$.
\end{itemize}
\end{enumerate}
\end{Th}

These different behaviours are illustrated in Figure \ref{fig_BBP_trans} below.

 \begin{figure}[h!]
\centering
\subfigure[Case where  $\theta=0.5$]{
\includegraphics[width=2.8in]{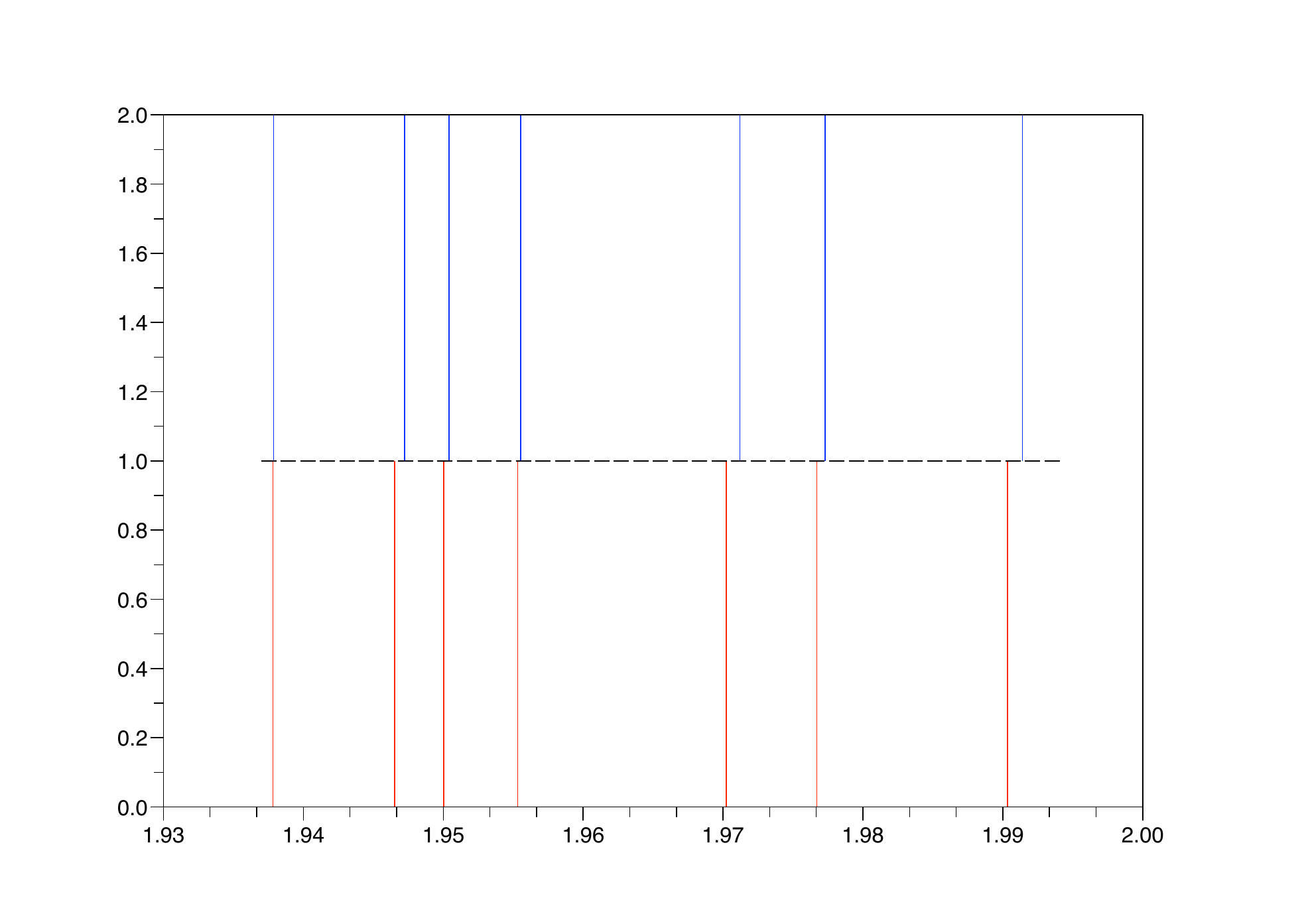}
\label{VPmax_GOE_def_below_seuil}} 
\subfigure[Case where $\theta=1.5$]{
\includegraphics[width=2.8in]{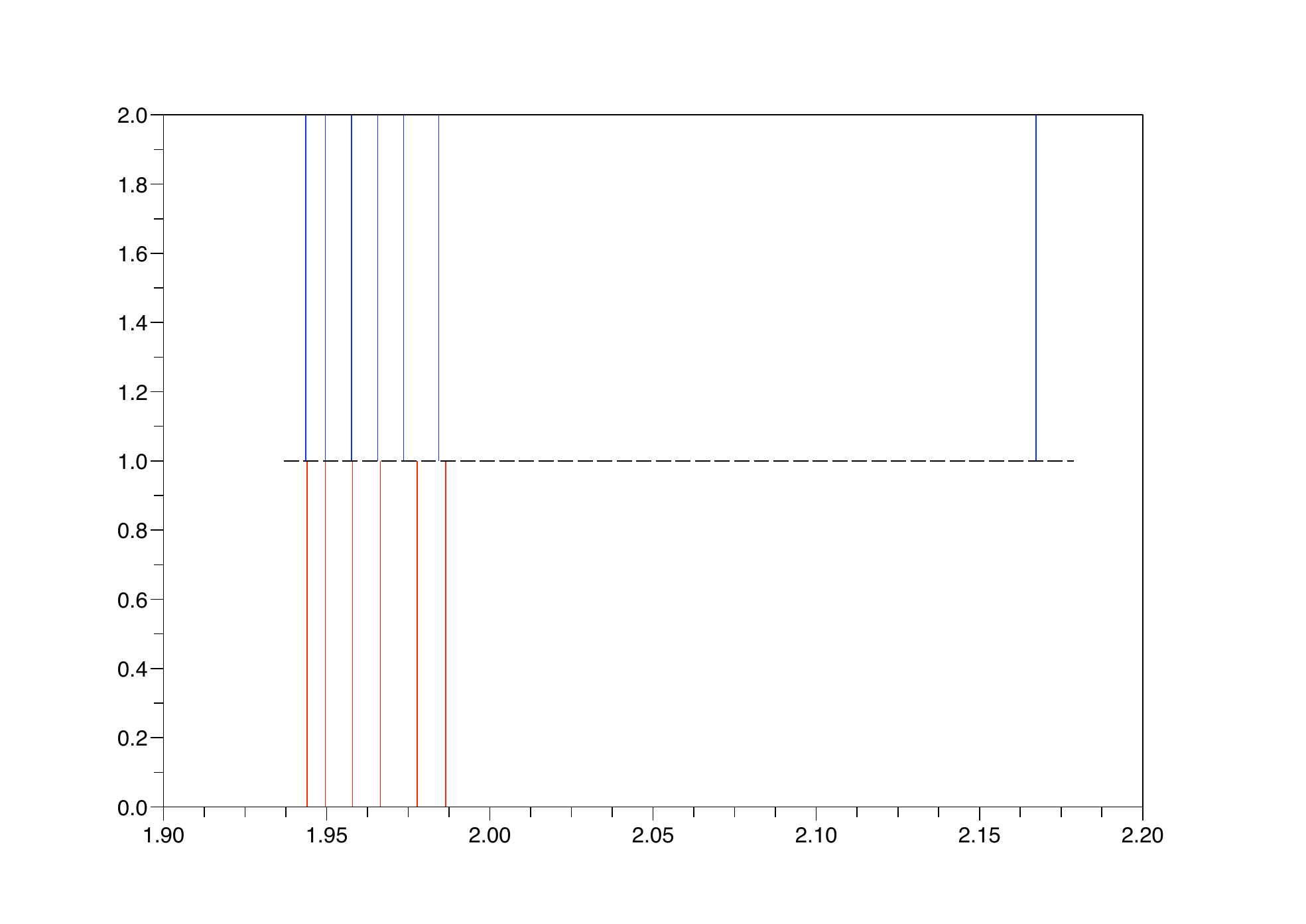}
\label{VPmax_GOE_def_below_seuil}
}

\caption{{\bf Comparison between the largest eigenvalues of a GUE matrix and those of the same matrix perturbed:}   the abscises of the vertical segments correspond to the largest eigenvalues of $X$, a GUE matrix with size $2.10^3$ (under the dotted line) or to those of    $\widetilde{X_n}=X+\operatorname{diag}(\theta, 0, \ldots, 0)$ (above the dotted line).   
   In the left picture, $\theta=0.5<\overline{\theta}=1$ and as predicted, $\widetilde{\lambda}_1\approx b=2$,  
  whereas in the right one,  $\theta =1.5>\overline{\theta}$, which indeed implies that $\widetilde{\lambda}_1\approx \rho_\theta =\theta+\frac{1}{\theta}=2.17$ 
  and  $\widetilde{\lambda}_2\approx b$.
  Moreover, in the left picture, we have, for all $i$,  $\widetilde{\lambda}_i\approx \lambda_i$, with some deviations  $$|\widetilde{\lambda}_i-\lambda_i|\ll \textrm{deviation of  $\lambda_i$ from its  limit $2$}.$$ In the same way, in the right picture,  $i$, $\widetilde{\lambda}_{i+1}\approx \lambda_i$, with some deviations $$|\widetilde{\lambda}_{i+1}-\lambda_i|\ll \textrm{deviation of  $\lambda_i$ from its  limit $2$}.$$
  At last, here, in the right picture, we have $\widetilde{\lambda}_1\approx 2.167$, which gives $\frac{\sqrt{n}(\widetilde{\lambda}_1-\rho_\theta)}{c_\theta}\approx 0.040$, reasonable value for a standard Gaussian variable.}
\label{fig_BBP_trans}
\end{figure}

The first part of this theorem will be proved in Section \ref{sec.away}, whereas Section \ref{sec.sticking} will
 be devoted to the study of the eigenvalues sticking to the bulk, i.e. to the proof of the second and third parts of the theorem.\\
Moreover, our results can be easily generalised to non-deterministic self-adjoint matrices $X_n$
that satisfy our hypotheses with probability tending to one. This will allow us to study in Section
\ref{secAppl} the deformations of various classical models. This will include
the study of the Gaussian fluctuations away from the bulk for rather general Wigner and  Wishart
matrices, 
hence providing a new proof of the first part 
of  \cite[Theorem 1.1]{FP07} and of  \cite[Theorem 3.1]{baiyao2005}
but also a new generalisation to non-white ensembles.  The study of
the eigenvalues that
stick to the bulk requires a finer control
on the eigenvalues of $X_n$ in the vicinity 
of the  edges of the bulk, which  we prove for 
random matrices such as Wigner  
and Wishart matrices with entries having a sub-exponential tail.
This result complements 
\cite[Theorem 1.1]{FP07},
where the  fluctuations of the largest eigenvalue
of a non-Gaussian Wishart matrix perturbed by a delocalised but deterministic
rank one perturbation  was studied. 
One should remark that our result depends 
very little on the law $\nu$ (only through its fourth moment in fact).

Our approach is based upon a determinant computation
(see  Lemma \ref{detf}),
which shows  that the eigenvalues of $\wtX$ we are interested in
are the 
solutions   of the equation 
\begin{equation}\label{eqint} f_n(z):= \det \lf(\lf[
G_{s,t}^{n}(z)
\ri]_{s,t=1}^{ r}-\diag(\theta_1^{-1},\ldots, \theta_r^{-1})\ri)=0,
\end{equation}
with
\be\label{135111}G_{s,t}^n(z):=\langle u^n_s,(z-X_n)^{-1} u^n_t\rangle,\ee
where $\langle \cdot, \cdot \rangle$ denotes the usual scalar product in $\C^n.$\\
By the law of large numbers for i.i.d. vectors, by \cite[Proposition 9.3]{benaych-rao.09} for uniformly distributed vectors or by applying
Theorem \ref{3110.23h27} (with $A^n=(z-X_n)^{-1}$), it is easy to
see that for any $z$ outside  the bulk,
$$\lim_{n\ra \infty} G_{s,t}^n(z)= \one_{s=t} G_{\mu_X}(z)$$
and hence  it is clear that one should expect the eigenvalues of $\wtX$
outside of the bulk to  converge to the solutions of $ G_{\mu_X}(z)=\theta_i^{-1}$ if they exist.
Studying the fluctuations of these eigenvalues amounts to analyse the
behavior of the solutions of  \eqref{eqint} around their limit. Such an approach 
was already developed in several papers (see e.g \cite{bai-yao-TCL} 
or \cite{CDF09}). However, to our knowledge,
the model we consider, with a fixed deterministic matrix $X_n$, was not yet studied 
and the fluctuations of the eigenvalues  which stick to the bulk of $X_n$ was never achieved
in such a generality.

For the sake of clarity, throughout the paper, we will call  
 ``hypothesis'' any hypothesis we need to make on the deterministic
part of the model $X_n$ and ``assumption'' any hypothesis we need to make  on the deformation $R_n.$\\
Moreover, because of concentration considerations that are developed in the Appendix of the paper, the proofs will be quite similar 
in the i.i.d. and orthonormalised models. Therefore, we will detail each proof in the i.i.d. model, which is simpler and then check
that the argument is the same in the orthonormalised model or detail the slight changes to make in the proofs.

{\bf Notations.} For the sake of clarity, we   recall here the main notations of the paper:\\ \\
$\bullet$ $\la_1^n\le \cdots\cdots\le \la_n^n$ are the eigenvalues of the deterministic   matrix $X_n$,\\ \\
$\bullet$ $\wtl_1^n\le \cdots\cdots\le\wtl_n^n$ are the eigenvalues of the perturbed   matrix $\wtX=X_n+\sum_{i=1}^r \tta_iu^n_iu_i^{n^*}$, where $r$ and the $\tta_i$'s are independent of $n$ and deterministic and  the column vectors $u_i^n$ are random and defined above,\\ \\
$\bullet$ 
$r_0\in \{0, \ldots, r\}$
 is \st $\theta_{1}\le \cdots \le\theta_{r_0}<0<\theta_{r_0+1}\le
\cdots\le \theta_r$,  \\ \\
$\bullet$ for $z$ out of the spectrum of $X_n$, $G^n_{s,t}(z)=\lan u^n_s, (z-X_n)^{-1} u^n_t\ran$,\\ \\
$\bullet$ for $z$ out of the support of $\mu$, $G_{\mu_X}(z)=\int\frac{1}{z-x}\ud\mu_X(x)$,\\ \\
$\bullet$ 
  $\overline\theta=\ff{\lim_{z\downarrow b}G_{\mu_X}(z)}\ge 0$ and $\underline\theta=\ff{\lim_{z\uparrow a}G_{\mu_X}(z)}\le 0$,\\ \\
$\bullet$ 
for any non null $\tta$, 
 $$\rho_{\theta}=\begin{cases}G_{\mu_X}^{-1}(1/\theta)&\textrm{if }\theta
\in (-\infty,\underline\theta)\cup (\overline\theta,+\infty),\\
a&\textrm{if  }\theta\in [\underline\theta,0),\\
b&\textrm{if   }{\theta}\in (0,\overline\theta],
 \end{cases}
$$
$\bullet$ 
$p_+$ is the number of $i$'s \st $\rho_{\theta_i}>b$,  $p_-$ is the number of $i$'s \st $\rho_{\theta_i}<a$ and $\alpha_1<\cdots<\alpha_q$ are the different values
of the $\theta_i$'s \st  $\rho_{\theta_i}\notin\{a,b\}$ (so that $q\le p_-+p_+$, with equality in the particular case where the $\tta_i$'s are pairwise distinct),\\ \\
$\bullet$  $\gamma_1^n, \ldots\ldots \gamma_{p_-+p_+}^n$ are the rescaled differences between the eigenvalues with limit out of $[a,b]$ and their limits:  \[  \gamma_i^n=\begin{cases}\sqrt{n}(\wtl_i^n-\rho_{\tta_i})&\mbox{if }\;1\le i\le p_-,\\ \\
\sqrt{n}(\wtl_{n-(p_-+p_+)+i}^n-\rho_{\tta_{r-(p_-+p_+)+i}})&\mbox{if }\; p_-<i\le p_-+p_+,\end{cases}
\]
$\bullet$ for any $x$ such that $x \le a$ or $x \ge b,$    $I_x$ is the set of indices $i$ such that 
$\rho_{\theta_i} = x$,\\ \\
$\bullet$ for any $j=1, \ldots, q$, $k_j$ is the number of indices $i$ \st $\tta_i=\al_j$, i.e. $k_j=|I_{\rho_{\al_j}}|$.

\section{Almost sure convergence of the extreme eigenvalues}
\label{sec.ps}
For the sake of completeness, in this section, we prove  
Theorem \ref{241109.17h50cor11}. In fact, we shall even prove the more general following result.

\begin{Th}\label{241109.17h50} Assume that Hypothesis \ref{hypspec} and Assumption \ref{hyponG} are satisfied. 

Let us fix, independently of $n$, an integer  $i\ge 1$   and   $V$,   a neighborhood of $\rho_{\tta_i}$ if $i\le r_0$  and of $a$ if $i>r_0$.
 Then $\wtl_i^{n}\in V$ with overwhelming probability. 
 
 The analogue result exists for largest eigenvalues:  for any fixed integer  $i\ge 0$   and   $V$, a neighborhood of $\rho_{\tta_{r-i}}$ if $i< r-r_0$  and of $b$ if $i\ge r-r_0$, $\wtl_{n-i}^{n}\in V$ with overwhelming probability.\end{Th}

By Lemma \ref{detf},   the eigenvalues of $\wtX$ which are not in the spectrum of $X_n$ are the solutions of the equation $$\det(M_n(z))=0,$$ with$$ M_n(z) =\lf[
G_{s,t}^{n}(z)
\ri]_{s,t=1}^{ r}-\diag(\theta_1^{-1},\ldots, \theta_r^{-1}),$$ the functions $G_{s,t}^n(\cdot)$ being defined in \eqref{135111}. For $z$ out of the support of $\mu_X$, let us introduce the $r\times r$ matrix $$ M(z) :=  \diag(G_{\mu_X}(z)-\theta_1^{-1},\ldots\ldots, G_{\mu_X}(z)-\theta_r^{-1}).$$ 
The key point, to prove Theorem \ref{241109.17h50}, is the following lemma.
For $A=[A_{i,j}]_{i,j=1}^r$ and $r\times r$ matrix, we set $|A|_\infty:=\sup_{i,j} |A_{i,j}|$. 

\begin{lem}\label{toto} Assume that Hypothesis \ref{hypspec} and Assumption \ref{hyponG} are satisfied. 
For any  $\delta, \eps>0,$ with overwhelming probability, $$\sup_{z, \, d(z, [a,b]) >\delta} |M(z)-M_n(z)|_\infty \le \eps.$$\end{lem}

In the case where the $\tta_i$'s are pairwise distinct, Theorem \ref{241109.17h50} follows directly from this lemma, because the $z$'s \st $\det(M(z))=0$ are precisely the $z$'s \st for some $i$, $G_{\mu_X}(z)=\ff{\tta_i}$ and because close continuous   functions on an interval have close zeros. The case where the $\tta_i$'s are not pairwise distinct can then be deduced
by an approximation procedure similar to the one of Section 6.2.3 of \cite{benaych-rao.09}.

\noindent{{\it Proof of Lemma \ref{toto}.}}
\textit{The i.i.d. model.} Fix $R$ \st for all $x\in [a-\delta/2, b+\delta/2]$ and $z\in \C$ with $|z|\ge R$, $$\lf|\ff{z-x}\ri|\le \f{\eps}{2}.$$ 

Then since the support of $\mu_X$ is contained in $[a,b]$ and for $n$ large enough, the eigenvalues of $X_n$ are all in $ [a-\delta/2, b+\delta/2]$, it suffices to prove that with overwhelming probability,  $$\sup_{|z|\le R, \, d(z, [a,b]) >\delta} |M(z)-M_n(z)|_\infty \le \eps.$$

Now, fix some 
$z$ \st $|z|\le R$, $d(z, [a,b]) >\delta,$ and $n$ large enough. By Proposition  \ref{hanson} with  $A=(z-X_n)^{-1}$, whose operator norm is
  bounded by $2\delta^{-1},$ we find 
 that for any $\e>0$, there exists $c>0$ such that
\be\label{18610.17h58}\Pro\left(\left|G_{s,t}^{n}(z)-1_{s=t}\frac{1}{n}\tr((z-X_n)^{-1})\right|\ge \frac{\delta^{-1}}{n^{1/2-\e}}\right) \le 4 e^{-cn^{2\e}}.\ee
It follows that there are $c, \eta>0$ \st for all $z$ \st $|z|\le R$, $d(z, [a,b]) >\delta,$ $$\Pro(|M(z)-M_n(z)|_\infty>\eps/2)\le e^{-cn^\eta}.$$ 
As a consequence, since the number of $z$'s \st $|z|\le R$ and $nz$ have integer real and imaginary parts has order $n^2$, 
there is a constant $C$ \st 
$$\Pro(\sup_z|M(z)-M_n(z)|_\infty   >\eps/2)\le Cn^2e^{-cn^\eta},$$ where the supremum is taken over complex numbers  $z=\f{k}{n}+i\f{l}{n}, $ with  $k,l\in \z$, \st $|z|\le R$, $d(z, [a,b]) >\delta$. 
Now, note that for $n$ large enough so that  the eigenvalues of $X_n$ are all in $ [a-\delta/2, b+\delta/2]$, the Lipschitz norm 
for $|\cdot|_\infty$ on the set of $z$'s \st $d(z, [a,b]) >\delta$ of the function $z\longmapsto M_n(z)$ is less than $\f{4}{\delta^2}.\max_{s,t=1 \ldots r} \|u_s^n\|\|u_t^n\|.$  Therefore, by Proposition \ref{hanson} again, with overwhelming probability $z\longmapsto M_n(z)$
is $\frac{4\sqrt n}{\delta^2}$-Lipschitz on this set.
The function $z\longmapsto M(z)$ is $\f{1}{\delta^2}$-Lipschitz on this set, so, with overwhelming probability, 
$$\sup_{|z|\le R, d(z, [a,b])>\delta }|M_n(z)-M(z)|_\infty\le
\max_{\substack{|z|\le R, d(z, [a,b])>\delta\\ nz\in \z+i\z} } |M_n(z)-M(z)|_\infty +8\delta^{-2} n^{-1/2}\, ,$$
which insures that for $n$ large enough, 
$$\Pro\left(\sup_{|z|\le R, d(z, [a,b])>\delta}|M(z)-M_n(z)|_\infty   >\eps\right)\le Cn^2e^{-cn^\eta}.$$This concludes the proof for the i.i.d. model.

\textit{ The orthonormalised model} can be  treated similarly, by writing $U_n=W^nG_n$
with $\sqrt{n} W^n$  a matrix converging almost surely to the identity by Proposition \ref{convm}.\hfill$\square$

\section{Fluctuations of the eigenvalues away from the bulk}
\label{sec.away}

\subsection{Statement of the results}

 Let $p_+$ be the number of $i$'s \st $\rho_{\theta_i}>b$ 
and $p_-$ be the number of $i$'s \st $\rho_{\theta_i}<a$. In this section, we study the fluctuations of the eigenvalues 
of $\wtX$ with limit out of the bulk, that is $
(\wtl_1^n,\ldots,\wtl_{p_-}^n,\wtl_{n-p_++1}^n,\ldots, \wtl_n^n)$.
We shall assume throughout this section that the spectral measure of $X_n$ converges to $\mu_X$
faster than $1/\sqrt n.$ More precisely,
\begin{hyp}\label{H2}
 For all $z\in \{\rho_{\al_1}, \ldots, \rho_{\al_q}\}$, $\sqrt{n}(G_{\mu_n}(z)-G_{\mu_X}(z))$ converges  to $0$.
\end{hyp}

Our  theorem deals with the limiting joint distribution of the   variables $\gamma^n_1, \ldots, \gamma^n_{p_-+p_+}$,  the rescaled differences between the eigenvalues with limit out of $[a,b]$ and their limits:   \[ \gamma_i^n=\begin{cases}\sqrt{n}(\wtl_i^n-\rho_{\tta_i})&\mbox{if }\;1\le i\le p_-\\ \\
\sqrt{n}(\wtl_{n-(p_-+p_+)+i}^n-\rho_{\tta_{r-(p_-+p_+)+i}})&\mbox{if }\; p_-<i\le p_-+p_+\end{cases}
\]

Let us recall that for $k\ge 1$, $\GOE(k)$ (resp. $\GUE(k)$) is the distribution of a $k\times k$ symmetric (resp. Hermitian) random matrix $[g_{i,j}]_{i,j=1}^k$ \st the random variables  
$\{\f{1}{\sqrt{2}} g_{i,i}\ste 1\le i\le k\}\cup\{  g_{i,j}\ste 1\le i<j\le k\}$
(resp. $\{  g_{i,i}\ste 1\le i\le k\}\cup\{\sqrt{2 }\Re(g_{i,j})\ste 1\le i<j\le k\}\cup\{\sqrt{2}\Im(g_{i,j})\ste 1\le i<j\le k\}$)
are independent standard Gaussian    variables.\\

The limiting behaviour of the eigenvalues with limit outside the bulk will depend on the law $\nu$ through the following quantity, called 
the {\it fourth cumulant} of $\nu$ $$\kappa_4(\nu):=\begin{cases}\int x^4\ud\nu(x)-3&\textrm{in the real case,}\\
\int |z|^4\ud\nu(z)-2&\textrm{in the complex case.}
\end{cases}$$ Note that if $\nu$ is Gaussian standard, then $\kappa_4(\nu)=0$.\\

The definitions of the $\al_j$'s and of the $k_j$'s have been given in Theorem \ref{mainint} and recalled in the Notations gathered at the end of the introduction above.
\begin{Th}\label{theogauss}
Suppose that Assumption \ref{hyponG} holds with $\kappa_4(\nu)=0,$ as well as
Hypotheses \ref{hypspec} and \ref{H2}. Then
 the law of  $$(\gamma^n_{\sum_{\ell=1}^{j-1}k_\ell+i}, 1\le i\le k_j)_{1\le j\le q}$$
 converges to the law of $(\lambda_{i,j}, 1\le i\le k_j)_{1\le j\le q},$ with $\lambda_{i,j}$
the $i$th largest eigenvalue 
of $c_{\alpha_j} M_{j}$ with $(M_1, \ldots, M_q)$ being independent
matrices, $M_j$ following the 
GUE$(k_j)$ (resp.   GOE$(k_j)$) distribution  if
$\nu$ is supported on  the complex plane (resp. the real line).
The constant $c_\alpha$ is given by
\be\label{28110.1}c_\alpha^2=\begin{cases}\f{1}{{\int (\rho_{\alpha}
-x)^{-2} \ud\mu_X(x)}}&\textrm{in the i.i.d. model,}\\ \\
\f{{\int  \f{\ud\mu_X(x)}{(\rho_{\alpha}
-x)^{2}}  -\ff{\alpha^{2}}}}{\lf(\int (\rho_{\alpha}
-x)^{-2} \ud\mu_X(x)\ri)^2}&\textrm{in the orthonormalised model.}
\end{cases}
\ee
\end{Th}

When $\kappa_4(\nu)\neq 0,$ we need a bit more than Hypothesis \ref{H2}, namely
 
\begin{hyp}
 \label{H2'}
For all $z\in \R\bck [a,b]$, there is a finite number $l(z)$ \st $$\begin{cases}
\ff{n}\sum_{i=1}^n((z-X_n)^{-1})_{i,i}^2\ninf l(z)&\textrm{in the i.i.d. model,}\\
\ff{n}\sum_{i=1}^n(((z-X_n)^{-1})_{i,i}-\ff{n}\Tr((z-X_n)^{-1}))^2\ninf l(z)&\textrm{in the orthonormalised model.}\end{cases}
$$
\end{hyp}

We then have a similar result.

\begin{Th}\label{942010.1}
In the case when Assumption \ref{hyponG} holds with $\kappa_4(\nu)\neq 0,$ under Hypotheses \ref{hypspec}, \ref{H2} and \ref{H2'},
Theorem \ref{theogauss} stays true, replacing the matrices $c_{\alpha_j}M_j$ by matrices $c_{\alpha_j}M_j+D_j$ where the $D_j$'s are independent diagonal random matrices, independent of the $M_j$'s, and \st for all $j$, the diagonal entries of $D_j$ are independent centred real Gaussian   variables, with variance $-l({\rho_{\al_j}})\kappa_4(\nu)/  G_{\mu_X}'(\rho_{\al_j})$.
\end{Th}

\subsection{Proof of Theorems \ref{theogauss} and  \ref{942010.1}}

We prove hereafter Theorem \ref{theogauss} and we will indicate briefly at the end of this section the minor changes to make 
to get Theorem \ref{942010.1}. The main ingredient will be a central limit theorem for quadratic forms, stated in Theorem  
\ref{3110.23h27} in the appendix.\\

For $i \in \{1, \ldots q\}$ and $x \in \R,$ we denote by $M^n(i,x)$ the $r \times r$ (but no longer symmetric) matrix
with entries given by

$$
\left[M^n(i,x)\right]_{s,t} := \left\{
\begin{array}{ll}
 \sqrt n \left(G^n_{s,t}\left(\rho_{\al_i}+\f{x}{\sqrt{n}}\right) - \one_{s=t}\frac{1}{\al_i}\right), & \textrm{if } s \in I_{\rho_{\al_i}},\\
G^n_{s,t}\left(\rho_{\al_i}+\f{x}{\sqrt{n}}\right) - \one_{s=t}\frac{1}{\theta_s}, & \textrm{if } s \notin I_{\rho_{\al_i}}.
\end{array}
\right.
$$

We set $\rho_n^i(x):= \rho_{\al_i}+\f{x}{\sqrt{n}}.$

The first step of the proof will be to get the asymptotic behavior of $M^n(i,x).$

 \begin{lem}\label{5110.23h51} Let $i \in \{1, \ldots q\}$ and $x \in \R$ be fixed.
 Under the hypotheses of Theorem \ref{theogauss}, $M^n(i,x)$ converges weakly, as $n$
goes to infinity, to the matrix $\mathcal M(i,x)$ with entries
\be \label{mix}
\left[\mathcal M(i,x)\right]_{s,t} := \left\{
\begin{array}{ll}
G_{\mu_X}^\prime(\rho_{\al_i}) (x \one_{s=t} - c_{\rho_{\al_i}} n_{s,t})
, & \textrm{if } s \in I_{\rho_{\al_i}},\\
\left(\frac{1}{\al_i} - \frac{1}{\theta_s}\right) \one_{s=t}, & \textrm{if } s \notin I_{\rho_{\al_i}},
\end{array}
\right.
\ee
  with $(n_{s,t})_{s,t = 1,\ldots, r}$ a family of independent Gaussian variables with
$n_{s,s} \sim \mathcal N(0,2)$ and $n_{s,t} \sim \mathcal N(0,1)$ when $s \neq t$ in the real case
(resp. $n_{s,s} \sim \mathcal N(0,1)$ and $\Re (n_{s,t}), \Im( n_{s,t}) \sim \mathcal N(0,1/2)$ and independent in the complex case).
 \end{lem}

\begin{proof}
 From \eqref{18610.17h58}, we know that for $s\notin I_{\rho_{\al_i}}$,
\be \label{28.12.10.1}
\lim_{n\ra\infty} \left[M^n(i,x)\right]_{s,t} =\left( G_{\mu_X}(\rho_{\al_i}) - \frac{1}{\theta_s}\right)\one_{s=t} = \left(\frac{1}{\al_i} - \frac{1}{\theta_s}\right) \one_{s=t}.
\ee

Let $s \in  I_{\rho_{\al_i}}.$ We write the decomposition 
$$M_{s,t}^n(i,x):=
\sqrt{n}\left(G^n_{s,t}(\rho_n^{i}(x))-\ff{\al_i }\one_{s=t}\right)=:
M_{s,t}^{n,1}(i,x)+
M_{s,t}^{n,2}(i,x)+M_{s,t}^{n,3}(i,x)
$$
where
\begin{eqnarray*}
M_{s,t}^{n,1}(i,x)&:=&\sqrt{n}\left( \langle u^n_s, (\rho_n^{i}(x)-X_n)^{-1} u^n_t\rangle
-\one_{s=t}\frac{1}{n}\tr((\rho_n^{i}(x)-X_n)^{-1})\right),\\
M_{s,t}^{n,2}(i,x)&:=&\one_{s=t}\sqrt{n}\left(\frac{1}{n}\tr((\rho_n^{i}(x)-X_n)^{-1})-\frac{1}{n}\tr((\rho_{\alpha_i}-X_n)^{-1})\right)
,\\
M_{s,t}^{n,3}(i,x)&:=&\one_{s=t}\sqrt{n} \left(\frac{1}{n}\tr((\rho_{\alpha_i}-X_n)^{-1})-G_{\mu_X}(
\rho_{\alpha_i})
\right).
\end{eqnarray*}

The asymptotics of the first term is given by  Theorem \ref{3110.23h27} with a variance given by
\be\label{lik0}
\lim_{n\ra\infty}\frac{1}{n}\tr((\rho_n^{i}(x)-X_n)^{-2})=
-G'_{\mu_X}(\rho_{\alpha_i}). \ee
As $\rho_{\alpha_i}$
is at distance of order one from the support of $X_n$,
we can expand $x/\sqrt{n}$ in $M_{s,t}^{n,2}(i,x)$ to deduce that
\be\label{lik}
\lim_{n\ra\infty}  M_{s,t}^{n,2}(i,x)=xG_{\mu_X}'(\rho_{\alpha_i})\one_{s=t}.\ee
Finally, by Hypothesis \ref{H2}, we have
\be\label{lik2}
\lim_{n\ra\infty} M_{s,t}^{n,3}(i,x)=0.\ee
Equations \eqref{28.12.10.1}, \eqref{lik0}, \eqref{lik} and \eqref{lik2} prove the lemma
(using the fact that the distribution of the Gaussian variables $n_{s,s}$ and $n_{s,t}$ are symmetric). 
\end{proof}

The next step is to study the behaviour of $(M^n(i,x))_{x\in \R}$ as a process on $\R.$
We will show in particular that the dependence in the parameter $x$ is very simple.
Let $(n_{s,t})_{s,t=1, \ldots, r}$ be a family of Gaussian random variables as in Lemma \ref{5110.23h51}
and define the random process $\mathcal M(i, \cdot)$ from $\R$ to $\mathcal M_n(\C)$ with
$\left[\mathcal M(i,x)\right]_{s,t} $ defined as in \eqref{mix} (where we emphasize that 
$(n_{s,t})_{s,t=1, \ldots, r}$ do not depend on $x).$ Then we have
\begin{lem}
 Let $i \in \{1, \ldots q\}$  be fixed. The random process  $ (M^n(i,x))_{x \in \R}$
converges weakly, as $n \ra \infty,$ to $\mathcal M(i, \cdot)$ in the sense of finite dimensional marginals.
\end{lem}

\begin{proof}
 This is a direct application of Remark \ref{3110.23h27-rmq}, as it is easy to check that for any $x, x^\prime \in \R,$
$$ \lim_{n \ra \infty}\frac{1}{n} \Tr \left(\left(\rho_{\al_i}+\f{x}{\sqrt{n}} -X_n\right)^{-1} -
\left(\rho_{\al_i}+\f{x^\prime}{\sqrt{n}} -X_n\right)^{-1} \right)^2 =0$$
\end{proof}

The last point to check is a result of asymptotic independence, from which the independence of the matrices 
$M_1, \ldots, M_q$ will be inherited. In fact, the matrices \linebreak $(M^n(1, x_1), \ldots,M^n(q, x_q)) $ won't be asymptotically independent
but their determinants will.
\begin{lem}
 For any $(x_1, \ldots, x_q)\in \R^q,$ the random variables
$$\mathrm{det} [M^n(1, x_1)], \ldots,\mathrm{det} [M^n(q, x_q)]$$ are asymptotically independent.
\end{lem}

\begin{proof}
 The key point is to show that,
\be \label{detexpansion}
 \mathrm{det}\left[ M^n(i, x)\right] = \mathrm{det} \left([M^n(i, x)]_{s,t \in I_{\rho_{\alpha_i}}}\right) \prod_{s\notin
I_{\rho_{\alpha_i}}} \left(\frac{1}{\alpha_i} - \frac{1}{\theta_s}\right)+ o(1),
\ee
where the remaining term is uniformly small  as $x$ varies in any compact of $\R.$\\

Then, as the set of indices $I_{\rho_{\alpha_1}}, \ldots, I_{\rho_{\alpha_q}}$ are disjoint, the submatrices 
involved in the main terms are independent in the i.i.d case and asymptotically independent in the orthonormalised case.\\

Let us now show \eqref{detexpansion}. Firstly, note that by the  convergence of $M^n_{s,t}(i,x)$ obtained in the proof of the 
Lemma \ref{5110.23h51}, we have
for all $s,t\in\{1, \ldots,r\}$ \st $s\neq t$ or $s\in I_{\rho_{\alpha_i}}$, for all $\kappa<1/2$, 
\be\label{241209.16h30}n^\kappa\lf(G^n_{s ,t}(\rho_n^i(x))-\one_{s=t}\ff{\theta_s}\ri)\ninf 0 \quad\textrm{(convergence in probability).}\ee
 
By  the formula \bes\label{23110915h092}  \mathrm{det}\left[ M^n(i, x)\right] = n^{\frac{k_i}{2}} 
\sum_{\sigma\in S_r}\operatorname{sgn}(\sigma)\prod_{s=1}^r
\lf(G^n_{s ,\sigma(s)}(\rho_n^i(x))-\one_{s=\sigma(s)}\ff{\theta_s}\ri),\ees
it suffices to prove that for any 
$\sigma\in S_r$ \st 
for some $i_0\in \{1, \ldots,r\}\bck I_{\rho_{\alpha_i}}$, $\sigma(i_0)\neq i_0$,  
\be\label{23110915h093}n^{\f{k_i}{2}}\prod_{s=1}^r\lf(G_{s ,\sigma(s)}^n(\rho_n^i(x))
-\one_{s=\sigma(s)}\ff{\theta_s}\ri)\ninf 0 \quad\textrm{(convergence in probability).}\ee
It   follows immediately  from \eqref{241209.16h30}  since
 for any $\kappa<1/2$,  in the above product, all the terms with index in $I_{\rho_{\alpha_i}}$ are of order at most $n^{-\kappa}$, giving
a contribution $n^{-k_i \kappa}$, and $i_0$ is not in $I_{\rho_{\alpha_i}}$ and satisfies $\sigma(i_0)\neq i_0$, yielding another term of order
 at most
$n^{-\kappa}$. Hence, the other terms being bounded because $\rho_n^i(x)$ stays bounded away from $[a,b]$,  the above product is at most of order $n^{-\kappa(k_i +1)}$ and so taking $\kappa\in (\f{k_i}{2( k_i+1)},\f{1}{2})$
proves \eqref{23110915h093}.
\end{proof}
Now as we have that, for $i \in \{1, \ldots, q\}$ and $x \in \R,$
$$  \mathrm{det}\left[ M^n(i, x)\right] = f_n\left(\rho_{\al_i}+\f{x}{\sqrt{n}}\right)n^{\frac{k_i}{2}},$$
we can deduce from the lemmata above the following
\begin{propo}
 Under the hypothesis of Theorem \ref{theogauss}, the random process
$$ \left(\left(n^{\frac{k_1}{2}}f_n\left(\rho_{\al_1}+\f{x}{\sqrt{n}}\right)\right)_{x \in \R}, \ldots,
\left(n^{\frac{k_q}{2}}f_n\left(\rho_{\al_q}+\f{x}{\sqrt{n}}\right)\right)_{x \in \R} \right)$$
converges weakly, as $n$ goes to infinity to the random process
$$ \left(\left( G^\prime_{\mu_X}(\rho_{\alpha_i})^{k_i} {\rm det}(xI - c_{\al_i}M_i) \prod_{s\notin
I_{\rho_{\alpha_i}}} \left(\frac{1}{\alpha_i} - \frac{1}{\theta_s}\right)\right)_{x\in \R}\right)_{1 \le i \le q}$$
in the sense of finite dimensional marginals, with the constants $c_{\al_i}$ and the joint distribution
of $(M_1, \ldots, M_q)$ as in the statement of Theorem  \ref{theogauss}.
\end{propo}
 From there, the proof of Theorem \ref{theogauss} is straightforward. 
\begin{proof}
 Let $$x_1(i)< y_1(i)<x_2(i)<y_2(i)<\cdots <y_{k_i}
(i)\, \qquad\textrm{($1\le i\le q$)},$$ be fixed.   Since, by Theorem \ref{241109.17h50}, for all $\eps>0$, for $n$ large enough, 
$f_n$ vanishes
exactly at $p_-+p_+$ points in $\R\bck[a-\eps, b+\eps]$,
 we have that
 \begin{multline*}
  \Pro\left[x_\ell(i)<\gamma^n_{\sum_{m=1}^{i-1} k_m+ \ell}< y_\ell(i), \qquad
\forall \ell=1, \ldots, k_i, \quad \forall i=1,\ldots q\right] \\
= \Pro \left[f_n\left(\rho_{\al_i}+\f{y_\ell(i)}{\sqrt{n}}\right)f_n\left(\rho_{\al_i}+\f{x_\ell(i)}{\sqrt{n}}\right)<0, \,\forall \ell = 1, \ldots, k_i,
\,\forall i=1, \ldots, q, \right]\\
\xrightarrow[n \ra \infty]{}  \Pro \left[\mathrm{det}\left(y_\ell(i)I-c_{\alpha_i}M_i\right)\mathrm{det}\left(x_\ell(i)I-c_{\alpha_i}M_i\right)
<0, \,\forall \ell = 1, \ldots, k_i,
\,\forall i=1, \ldots, q, \right]\\
= \Pro \left[x_\ell(i)<\lambda_{i,\ell}< y_\ell(i)
, \,\forall \ell = 1, \ldots, k_i,
\,\forall i=1, \ldots, q, \right]
 \end{multline*}
\end{proof}

To prove Theorem \ref{942010.1}, the only substantial change to make is in the definition \eqref{mix},
in the case when $s \in I_{\rho_{\alpha_i}},$ we have to put 
$$\left[\mathcal M(i,x)\right]_{s,t} := 
G_{\mu_X}^\prime(\rho_{\al_i}) (x \one_{s=t} - c_{\rho_{\al_i}} n_{s,t}) - \kappa_4(\nu)l(\rho_{\al_i}). $$
The convergence of $\left[ M^n(i,x)\right]_{s,t}$ to $\left[\mathcal M(i,x)\right]_{s,t}$
  is again obtained by applying Theorem \ref{3110.23h27}.

\section{The sticking eigenvalues}
\label{sec.sticking}
\subsection{Statement of the results}
To study the fluctuations of the eigenvalues
which stick to the bulk, we need a more precise 
information on the eigenvalues of $X_n$
in the vicinity of their extremes. 
 More explicitly,
we shall need the following 
additional hypothesis, which depends 
on  a  positive integer  $p$ and a real number $\alpha\in (0,1)$.
Note that this hypothesis has two versions: Hypothesis \ref{H3a}$[p, \alpha,a]$ is  adapted to the study of the smallest eigenvalues (it is the version detailed below) and Hypothesis \ref{H3a}$[p, \alpha,b]$ is adapted to the study of the largest eigenvalues (this version is only outlined below).

\begin{hyp}
 \label{H3a}{$[p, \alpha,a]$}
There exists a sequence $m_n$ of positive integers 
tending to infinity   \st  $m_n=O(n^\alpha)$, \be\label{16511}\liminf_{n\to\infty} \ff{n}\sum_{i=m_n+1}^n\ff{\la_p^n-\la_i^n}\ge \ff{\udl{\tta}},\ee and there exist 
 $\eta_2>0$ and $\eta_4 >0$, 
so that   for $n$ large enough 
\be
\label{ecarts_vp_190110}  
 \sum_{i=m_n+1}^{n} \ff{(\la_p^{n}-\la_{i}^n)^{2}} \le
n^{2-\eta_2},
\ee
\be
  \textrm{ and } \quad\sum_{i=m_n+1}^n \frac{1}{(\la_p^{n}-\la_{i}^n)^4}\le n^{4-\eta_4}\,.\label{norankone}
\ee
\end{hyp}

\noindent{\bf Hypothesis  \ref{H3a}. }{\it $[p, \alpha,b]$ is the same hypothesis where we
replace $\la_p^{n}-\la_{i}^n$  by $\lambda_{n-p+1}^n-\lambda_{n-i+1}^n$,  
and  \eqref{16511} becomes
$$  \limsup_{n\to\infty}\ff{n}\sum_{i=m_n+1}^n\ff{\la_{n-p+1}^{n}-\la_{n-i+1}^n}\le \ff{\ovl{\tta}}\,.$$}

For many matrix models, the behaviors of largest and smallest eigenvalues are similar, and Hypothesis  \ref{H3a}  $[p, \alpha,a]$ is satisfied \ssi 
Hypothesis  \ref{H3a} $[p, \alpha,b]$ is satisfied. In such cases, we shall simply say that {\bf Hypothesis  \ref{H3a}} $[p, \alpha]$ is satisfied.

For rank one perturbations and in the i.i.d. model, we will only require the two first conditions \eqref{16511} and \eqref{ecarts_vp_190110}
whereas for higher rank perturbations, we will need in addition \eqref{norankone} to control the off-diagonal terms
of the determinant.\\

Moreover,
we shall not study the critical case where for some $i$, $\theta_i\in\{\udl{\tta},\ovl{\tta}\}$.

\begin{assum}
 \label{hypthetamin}
For all $i$, $\tta_i\ne \udl{\tta}$ and  $\tta_i\ne 
\ovl{\tta}$.
\end{assum}

In fact, Assumption \ref{hypthetamin} can be weakened into: for all $i$, $\tta_i\ne \udl{\tta}$ (resp. $\tta_i\ne 
\ovl{\tta}$) if we only study the smallest (resp. largest) eigenvalues.

The fact that the eigenvalues of the non-perturbed 
matrix are sufficiently spread at the edges to insure the above hypothesis
allow the eigenvalues of the perturbed matrix to
be very close to them, as stated in the following theorem.
 
\begin{Th}\label{theostick}
Let $I_a=\{i\in[1, r]: \rho_{\theta_i}=a\}=[p_-+1,r_0]$ (resp. $I_{b}=
\{i\in[1, r]: \rho_{\theta_i}=b\}=[r_0+1,r-p_+]$)  be the set of indices
corresponding to the  eigenvalues $\wtl_i^n$ (resp. $\wtl_{n-r+i}^n$)
converging to the lower (resp. upper) bound
of the support of $\mu_X$. Let us suppose 
 Hypothesis \ref{hypspec},  Hypothesis \ref{H3a} $[r, \alpha,a]$ (resp. Hypothesis \ref{H3a} $[r, \alpha,b]$) and  Assumptions \ref{hyponG} and \ref{hypthetamin} to hold. Then  for any 
$\alpha^\prime > \alpha,$  we have, for all $i\in I_a$ (resp. $i\in I_b$), 
\beq&& \qquad\quad
 \min_{1\le k\le i+ r-r_0}
|\wtl_{i}^n-\la_{k}^n|\le n^{-1+\alpha^\prime},\\
&&\textrm{(resp. } \min_{n-r+i-r_0\le k\le n}
|\wtl_{n-r+i}^n-\la_{k}^n|\le n^{-1+\alpha^\prime}\textrm{)}\eeq
with overwhelming probability.
\end{Th}
Moreover, in the case where the perturbation  has rank one,
we can locate exactly in the neighborhood of
which eigenvalues of the non-perturbed matrix the eigenvalues
of the perturbed matrix lie.

We state hereafter the result for the smallest eigenvalues, but of
course a similar statement holds for the largest ones.
\begin{Th}\label{exact1d}
Let $(\widetilde\lambda_i^n)_{i\ge 1}$ be the
eigenvalues of $X_n+\theta u_1u_1^*$, with $\tta<0$.  Then, under Assumption \ref{hyponG} and Hypothesis \ref{hypspec},
if 
 \eqref{16511} and \eqref{ecarts_vp_190110} in Hypothesis   \ref{H3a}  $[p,\alpha,a]$ hold for some $\alpha\in (0,1)$
and  a  positive integer $p $,  then  for any 
$\alpha^\prime > \alpha$, we have  
\begin{itemize}
\item[(i)]  if $\theta<\udl{\tta}$, 
$\widetilde\lambda_1^n$ converges to $\rho_\theta<a$
whereas $n^{1-\alpha'}(\widetilde\lambda_{i+1}^n-\lambda_{i}^n)_{1\le i\le p-1}$
vanishes in probability as $n$ goes to infinity,\\
\item[(ii)] if $\theta\in (\udl{\tta}, 0)$,
$n^{1-\alpha'}(\widetilde\lambda_{i}^n-\lambda_{i}^n)_{1\le i\le p}$
vanishes in probability as $n$ goes to infinity,\\
\item[(iii)] if, instead of \eqref{16511} and \eqref{ecarts_vp_190110} in Hypothesis   \ref{H3a}  $[p,\alpha,a]$, one supposes  \eqref{16511} and \eqref{ecarts_vp_190110} in Hypothesis   \ref{H3a}  $[p,\alpha,b]$ to hold,  then
$n^{1-\alpha'}(\widetilde\lambda_{n-i}^n-\lambda_{n-i}^n)_{0\le i< p}$
vanishes in probability as $n$ goes to infinity.
\end{itemize}
\end{Th}

\begin{Th}\label{corstick} Consider the i.i.d. model and let $(\widetilde\lambda_i^n)_{i\ge 1}$ be the
eigenvalues of $X_n+\sum_{i=1}^r\theta_i u_iu_i^*$. Let $p_-$ (resp. $p_+$) be
the number of indices $i$ so that $\rho_{\theta_i}<a$ (resp. $\rho_{\theta_i}>b$).
We
assume that  Assumptions \ref{hyponG} and \ref{hypthetamin},  Hypothesis \ref{hypspec}, and   \eqref{16511} and \eqref{ecarts_vp_190110} in 
  Hypotheses   \ref{H3a}  $[p,\alpha,a]$ and $[q,\alpha,b]$ hold for some $\alpha\in (0,1)$
and     integers $p , q$. Then,  for all  $\al'>\al$, for all  fixed  $1\le i\le p-(p_-+r)$ and $0\le j<p-(p_++r)$,  $$n^{1-\al'}(\widetilde\lambda_{p_-+i}^n-\lambda^n_{i})\qquad\textrm{ and }\qquad n^{1-\al'}(\widetilde\lambda_{n-(p_++j)}^n-\lambda^n_{n-j})$$ both vanish in \pro as $n$ goes to infinity.
\end{Th}

Note that if $ p-(p_-+r)\le 0$ (resp. if $p-(p_++r)<0$), then the statement of the theorem is empty as far as $i$'s (resp. $j$'s) are concerned. The same convention is made throughout the proof. 


\subsection{Proofs}
Let us first prove Theorem  \ref{theostick}. Let us choose $i_0\in I_a$ and study the behaviour of $\wtl_{i_0}^n$
(the case of the largest eigenvalues can be treated similarly). We assume throughout the section that Hypotheses \ref{hypspec}, \ref{H3a} $[r, \al, a]$ and  Assumptions \ref{hyponG} and \ref{hypthetamin}
are satisfied. We also fix $\alpha^\prime > \alpha.$

We  know, by Lemma \ref{detf},
 that the  eigenvalues    of $\wtX$ which are not eigenvalues of $X_n$ are the $z$'s \st  \be\label{6110.11h31}
  \det( M_n(z))=0,  \ee  where \be\label{6110.11h31.bis0511}M_n(z)=\lf[
G_{s,t}^{n}(z)
\ri]_{s,t=1}^{ r}-\diag(\theta_1^{-1},\ldots, \theta_r^{-1})\ee and  for all $s,t$, 
$$G_{s,t}^n(z)=\langle u^n_s, (z-X_n)^{-1} u^n_t\rangle.$$

Recall that by Weyl's interlacing inequalities (see \cite[Th. A.7]{alice-greg-ofer})
$$\widetilde \lambda^n_{i_0}\le \lambda^n_{i_0+r-r_0}.$$

Let  $\zeta$ be a fixed constant \st $\max_{1\le i\le p_-}  \rho_{\theta_i} <\zeta<a$. By Theorem \ref{241109.17h50}, we know that
\begin{lem}\label{8210.1h06-zeta} With overwhelming probability, $\widetilde \lambda^n_{i_0}>\zeta$. \end{lem}

We want to show that \eqref{6110.11h31}
is not possible on $$\Omega_n:= \left\{ z\in[\zeta, \lambda^n_{i_0+r-r_0}]\ste  \min_{1\le k\le i_0+ r-r_0}
|z-\la_{k}^n|> n^{-1+\alpha^\prime}
\right\}.$$

The following lemma deals with the asymptotic behaviour of the {\it off-diagonal terms} of the matrix $M_n(z)$ of \eqref{6110.11h31.bis0511}.
\begin{lem}\label{8210.1h06}For $s\neq t$ and  $\kappa>0$ small enough,
\be\label{bon15511}\sup_{z\in \Omega_n
}|G_{s,t}^{n}(z)| \le n^{-\kappa}
 \ee
with overwhelming probability.
\end{lem}

The following lemma deals with the asymptotic behaviour of the {\it diagonal terms} of the matrix $M_n(z)$ of \eqref{6110.11h31.bis0511}.
\begin{lem}\label{8210.00h29} For all $s=1, \ldots, r$, for all  $\delta>0$, 
any $\delta>0$,
 \begin{equation}\label{bon_prec} \sup_{z\in\Omega_n}\lf|G_{s,s}^{n}(z)-\ff{\udl{\tta}}\ri|\le  \delta
 \end{equation}

with overwhelming probability.  
\end{lem}
Let us assume these lemmas proven for a while and complete
the proof of Theorem  \ref{theostick}. By
 these two lemmas, for $z\in \Omega_n$,
 we find by expanding the
determinant that with overwhelming probability,
\be\label{1651111h30}\det (M_n(z))=\prod_{s=1}^r \left(G^n_{s,s}(z)-\frac{1}{\theta_i}\right)
+O(n^{-\kappa}),\ee where the $O(n^{-\kappa})$ is uniform on $z\in \Omega_n$.
Indeed, in the second term of the right hand side of $$\det (M_n(z))=\prod_{s=1}^r \left(G^n_{s,s}(z)-\frac{1}{\theta_i}\right)+\sum_{\sigma\in S_r\bck\{Id\}}
\operatorname{sign}(\si)\prod_{s=1}^r (G_{s, \si(s)}^n(z)-\one_{s=\si(s)}\tta_s^{-1}),$$ each diagonal term is bounded  and each non diagonal term is $O(n^{-\kappa})$.

Since for all $i$, $\tta_i\ne \udl{\tta}$, \eqref{1651111h30} and Lemma \ref{8210.00h29} allow to assert that with  overwhelming probability, for all $z\in \Omega_n$, $\det (M_n(z))\ne 0$. 
It completes the proof of the theorem. \qed

We finally prove the two last lemmas. 

\noindent
{\it Proof of Lemma \ref{8210.1h06}.}
Let us consider $z\in \Omega_n $  ($z$ might depend on $n$, but for notational brevity, we omit to denote it by $z_n$). 
We treat simultaneously the orthonormalised model
and the i.i.d. model (in the i.i.d. model, one just takes $W^n=I$ and replaces $\|( G^n(W^n)^T)_s\|_2$ by $\sqrt{n}$ in the proof below).
Observe that if we write $X_n=O^*D_nO$ with $D_n=(\lambda_1^n,\ldots,\lambda_n^n)$ and $O$ a unitary or orthogonal matrix,
\beq G_{s,t}^{n}(z)&=&\langle u^n_s, (z-X_n)^{-1}
u^n_t\rangle\\
&=& \sum_{l=1}^n \frac{\ovl{(Ou^n_s)_l} {(Ou^n_t)_l}}{z-\la_l^n}\eeq
The first step is to show that for any $\e>0$, with overwhelming probability,
\be\label{contou}\max_{l\le n, s\le r}
|(Ou^n_s)_l|\le n^{-\frac{1}{2}+\e}.\ee
Indeed, with $O_l$ the $l$th
row vector of $O$ and using the notations of Section  \ref{section.concentration.estimates},
$$(Ou^n_s)_l=\langle O_l, u^n_s\rangle=\frac{1}{\|( G^n(W^n)^T)_s\|_2}
\sum_{t=1}^r W^n_{s,t} \langle O_l, g^n_t\rangle.$$
But $g\mapsto \langle O_l, g^n_s\rangle$ is Lipschitz 
for the Euclidean norm with
constant one. Hence,
by concentration inequality due to
the log-Sobolev hypothesis (see e.g. \cite[section 4.4]{alice-greg-ofer}), there exists $c>0$ such that
for all $\delta>0$,
$$\Pro\left(|\langle O_l, g^n_s\rangle|>\delta\right)
\le 4e^{-c\delta^2}$$
so that
$$\Pro\left(\max_{l \le n, s\le r}
|\langle O_l, g^n_s\rangle|\ge  n^\e \right)\le  4n^4e^{-c n^{2\e}}.$$
From Proposition \ref{convm}, we know that  with overwhelming probability, 
$\|( G^n(W^n)^T)_s\|_2$
is bounded below by  $\sqrt{n}n^{-\e}$ and the entries of $W^n$ are of
order one. This gives therefore  \eqref{contou}.

We now make the following decomposition
\beq G_{s,t}^{n}(z)
&=& \underbrace{\sum_{l=1}^{m_n}\f{\ovl{(Ou_{s}^n)_l}(Ou_{t}^n)_l}{z -\la_l^n}}_{:=A_n(z)}+ \underbrace{\sum_{l=m_n+1}^{n}\f{\ovl{(Ou_{s}^n)_l}(Ou_{t}^n)_l}{z-\la_l^n}}_{:=B_n(z)}.
\eeq

Note that as   $|(Ou_{s}^n)_l|, 1\le l\le m_n$,
 are smaller than $n^{-\frac{1}{2} +\e'}$ by
\eqref{contou}, for any $\e'>0$,
 with overwhelming probability, we have,   uniformly on  $z\in\Omega_n$,
\beq\label{c1}
\left| A_n(z)\right|\le  m_n n^{1-\alpha'} n^{-1+2\e'}
= O(n^{\alpha-\alpha'+2\e'})
\eeq
We choose $0<\e'\le(\alpha'-\alpha)/4$ and now study $B_n(z)$ which can be written
$$B_n(z)=\langle u^n_s, P(z-X_n)^{-1}P u^n_t\rangle$$
with $P$ the orthogonal
projection onto the linear span of the eigenvectors of $X_n$ 
corresponding to the eigenvalues $\lambda_{m_n+1}^n,\ldots, \lambda_{n}^n$. By the second point in Proposition \ref{hanson}, with $z\in\Omega_n$,
for all $s\neq t$, 
\begin{multline*}
\Pro\left(\left|\langle g^n_s, P(z-X_n)^{-1}P g^n_t\rangle 
\right|
\ge \delta\sqrt{\tr(P(z-X_n)^{-2})+\kappa\sqrt{
\tr(P(z-X_n)^{-4})}}\right)\\ \le 4e^{-c\delta^2}+4e^{-c\min(\kappa, \kappa^2)}.
\end{multline*}
Moreover, by
 Hypothesis \ref{H3a}, for $n$ large enough,  for all $z\in \Omega_n$, 
 $$\tr(P(z-X_n)^{-2})\le n^{2-\eta_2}\textrm{ and }\tr(P(z-X_n)^{-4})\le n^{4-\eta_4}.$$
 We deduce
that there is $C, \eta>0$ \st for all $z\in \Omega_n$, 
\be\label{qwe}
\Pro\lf(\left|\frac{1}{n}\langle g^n_s, P(z-X_n)^{-1}P g^n_t\rangle
\right|> n^{-\frac{\eta_2\wedge \eta_4}{8}}\ri)\le Ce^{-n^{\eta}}
\ee
A similar control is verified for $s=t$ since we have, 
by Proposition \ref{hanson},
\be\label{poiu}
\Pro\left( \left|\frac{1}{n}\langle g_{s}, P(z-X_n)^{-1}P  g_{s}\rangle
-\frac{1}{n}\tr\left(P(z-X_n)^{-1}\right)\right|\ge \delta\right)
\le 4 e^{-c\min\{\delta^2n^{\eta_2},\delta n^{\eta_2/2}\}},\ee
whereas Hypothesis \ref{H3a} insures that
the term $\frac{1}{n}\tr( P(z-X_n)^{-1})
$ is bounded uniformly on $\Omega_n$. Thus, up to a change of the constants $C$ and $\eta$, there is a constant $M$ \st for all $z\in \Omega_n$,
 $$\Pro\lf(\lf|\frac{1}{n}\langle g_{s}, P(z-X_n)^{-1}P  g_{s}\rangle\ri|\ge M\ri)\le Ce^{-n^\eta}.$$
Therefore,  with Proposition \ref{convm} and developing
the vectors $u_s^n$'s as the normalised column vectors of  $G^n(W^{n})^{T}$, we conclude that, up to a change of the constants $C$ and $\eta$, for all $z\in \Omega_n$,
\be\label{contb2}
\Pro\lf(|B_n(z)| \ge n^{-\frac{\eta_2\wedge \eta_4}{8}}\ri)\le Ce^{-n^\eta}.
\ee
Hence, we have proved that there exists
$\kappa>0, C$ and $\eta >0$ so that  for all $z\in \Omega_n$,
$$\Pro\left(\left|G^n_{s,t}(z)\right|\ge n^{-\kappa}\right)\le Ce^{-n^{\eta}}.$$
We finally obtain this control uniformly on $z\in \Omega_n$
by noticing that $z\ra G^n_{s,t}(z)$ is Lipschitz
on $\Omega_n$, with constant bounded by $(\min|z-\lambda_i|)^{-2}\le
n^{2-2\alpha'}$.
Thus, if we take a grid $(z_k^n)_{0\le k\le cn^2}$
 of $\Omega_n$ with mesh $\le n^{-2+2\alpha'-\kappa}$ (there are about $n^2$ such
$z_k^n$'s)
 we have
$$\sup_{z\in\Omega_n}
\left|G^n_{s,t}(z)\right|\le \max_{1\le k\le
 cn^2} \left|G^n_{s,t}(z_k^n)\right| +n^{-\kappa}.$$
Since there are at most $cn^2$ such $k$  and $n^2$ possible $i,j$,
 we
conclude that
$$\Pro\left(\sup_{z\in \Omega_n}|G^n_{s,t}(z)|\ge
2n^{-\kappa}\right)\le  c^2n^4 Ce^{-n^{\eta}}$$
which completes the proof.\qed

\noindent{\it Proof of Lemma \ref{8210.00h29}.} We shall use the decomposition\be\label{17511.10h42}
G^n_{i,i}(z)=\langle u^n_i, P(z-X_n)^{-1}P
u^n_i\rangle+ \langle u^n_i,(1-P)(z-X_n)^{-1}(1-P)
u^n_i\rangle,\ee with $P$ as above the orthogonal
projection onto the linear span of the eigenvectors of $X_n$ 
corresponding to the eigenvalues $\lambda_{m_n+1}^n,\ldots, \lambda_{n}^n$, and then prove that for $z\in \Omega_n$, $$\langle u^n_s, P(z-X_n)^{-1}P
u^n_s\rangle\approx \ff{\udl{\tta}},$$ whereas \beq \langle u^n_s,(1-P)(z-X_n)^{-1}(1-P)
u^n_s\rangle&\le& \underbrace{\ff{\min_{1\le k\le m_n}|z-\la_k^n|}}_{\le n^{1-\al'}}  \underbrace{\|(1-P)
u^n_s\|_2^2}_{\approx n^{-1}\operatorname{rank}(1-P)}\\ &\approx&  n^{-\al'}m_n=O(n^{\al-\al'})=o(1).\eeq

Let us now give a formal proof. Again,
we first prove the estimate for a fixed $z\in \Omega_n$,
the uniform estimate on $z$ being  obtained by a grid
argument as in the previous proof (a key point being that the constants $C$ and $\eta$ of the definition of {\it overwhelming probability}
 are independent of the choice of $z\in \Omega_n$).

First, observe that  \eqref{16511} implies that for any sequence $\eps_n$ tending to zero, 
\be\label{18610.10h48}\lim_{n\to\infty}\sup_{\substack{a-\eps_n\le z \le \la_p^n}}\lf| \ff{n}\sum_{i=m_n+1}^n\ff{z-\la_{i}^n}- \ff{\udl{\tta}}\ri|= 0.\ee
Indeed, for all $\e>0$,  for $n$ \st $\la_n^p$ and $a-\eps_n$ are both $\ge a-\e$, we have,  for all $z\in [a-\eps_n, \la_p^n]$, 
$$ \ff{n}\sum_{i=m_n+1}^n\ff{\la_p^{n}-\la_{i}^n}\le\ff{n}\sum_{i=m_n+1}^n\ff{z-\la_{i}^n}\le\ff{n}\sum_{i=m_n+1}^n\ff{a-\e-\la_{i}^n},$$
so that \eqref{16511} and 
$$
\lim_{\e\downarrow 0 }\limsup_{n\ra\infty}
 \ff{n}\sum_{i=m_n+1}^n\ff{a-\e-\la_{i}^n}=\lim_{\e\downarrow 0 }G_{\mu_X}(a-\e)= \ff{\udl{\tta}}$$
imply \eqref{18610.10h48}.

 So let us consider $z\in \Omega_n$ ($z$ might depend on $n$, but for notational brevity, we omit to denote it by $z_n$). By the inequality  
$|z-\lambda_k^n|> n^{-1+\alpha'}$ for all $1\le k\le m_n$ and \eqref{17511.10h42}, 
we have \be\label{16511.DSKbye}
\lf|G^n_{s,s}(z)-\langle u^n_s, P(z-X_n)^{-1}P
u^n_s\rangle \ri|\le n^{1-\alpha'} \|(1-P)u^n_s\|_2^2.\ee
But  as in the 
previous proof,
we have 
$$\langle u^n_s, P(z-X_n)^{-1}P
u^n_s\rangle=\frac{n}{\|(G^n(W^n)^T)_s\|_2^2}\sum_{t,v=1}^{s}
W_{s,v}^n\ovl{W_{s,t}^n} \frac{1}{n}\langle g^n_t, P(z-X_n)^{-1}P
g^n_v\rangle$$
with, by \eqref{qwe}, the off diagonal terms  $t\neq v$
of order $n^{-\eta_2\wedge\eta_4/8}$ with overwhelming probability,
whereas the diagonal terms are close to $ \frac{1}{n}\tr(
P(z-X_n)^{-1})$ with overwhelming probability
by \eqref{poiu}.
Hence, we deduce with 
Proposition \ref{hanson} that 
for any $\delta>0$,
$$\left|\langle u^n_s, P(z-X_n)^{-1}P
u^n_s\rangle- \frac{1}{n}\tr(P((z-X_n)^{-1}))\right|\le \delta$$
with overwhelming  probability.
Hence, by \eqref{18610.10h48}, for any $\delta>0$ 
\be \label{eqlemma4.8}
\lf|\langle u^n_s, P(z-X_n)^{-1}P
u^n_s\rangle- \frac{1}{\underline{\theta}}\ri|\le \delta
\ee
with overwhelming probability.
On the other hand 
$$\|(1-P)u^n_s\|_2^2=\frac{1}{\|(G^n(W^n)^T)_s\|_2^2}\sum_{t,v=1}^r
W_{s,t}^n\ovl{W_{s,v}^n} \langle (1-P)g^n_t,(1-P)g^n_v\rangle$$
By Proposition \ref{convm}, the denominator is
of order $n$ with overwhelming probability, whereas
by Proposition \ref{hanson}, the numerator
is of order 
$m_n+n^\e\sqrt{m_n}$ (since $\tr(1-P)=m_n$) with overwhelming probability.
As $W^n$ is bounded by Proposition \ref{convm}
we conclude that
\be\label{16511.DSKbye2}\|(1-P)u^n_s\|_2^2\le 2\frac{m_n}{n}\ee
with overwhelming probability. 
Putting together Equations \eqref{16511.DSKbye},  \eqref{eqlemma4.8} and \eqref{16511.DSKbye2}, 
we have proved that 
for any $z\in\Omega_n$, any $\delta>0$, 
$$\lf|G^n_{s,s}(z)- \frac{1}{\underline{\theta}}\ri|\le \delta$$
with overwhelming probability, the constants $C$ and $\eta$ of the definition of {\it overwhelming probability} being independent of the choice of $z\in \Omega_n$
We do not detail the grid argument used to get a control uniform on $z$ because this argument   is similar to what we did in the proof of the previous lemma. \qed

  \noindent{\it Proof of Theorem \ref{exact1d}}.
In the one dimensional case, the eigenvalues 
of $\widetilde X_n$ which do not belong to
the spectrum of $X_n$ 
 are the zeroes of
\begin{equation}\label{mn}
f_n(z)=\frac{1}{n}\langle g, (z-X_n)^{-1} g\rangle -\varepsilon_n(g) \frac{1}{\theta}
\end{equation}
with $\varepsilon_n(g)=1$ or $\|g\|_2^2/n$ according to
the model we are considering. 
A straightforward study of the function $f_n$ tells us that the eigenvalues of $\widetilde X_n$ are  distinct from those
of $X_n$ as soon as $X_n$ has no multiple eigenvalue and $$(\textrm{matrix of the eigenvectors of $X_n$})^*\times g$$ has no null entry,
which we can always assume up
to modify $X_n$ and $g$ so slightly that the fluctuations
of the eigenvalues are not affected. We do not detail these arguments but the reader can refer 
to Lemmas 9.3, 9.4 and 11.2 of \cite{BGMpgd} for a full proof in the finite rank case.\\ 
Therefore, 
\eqref{mn} characterises all the eigenvalues of $\widetilde X_n$.
Moreover, by Weyl's interlacing properties, for $\theta<0$, 
$$\widetilde\lambda_1^n<\lambda_1^n<\widetilde\lambda_2^n<\lambda_2^n<\cdots<
\widetilde \lambda_n^n<\lambda_n^n\,.$$
Theorems \ref{241109.17h50} and \ref{theostick} thus already settle the
study of $\widetilde\lambda_1^n$ which either goes to $\rho_\theta$ or is at distance $O(n^{-1+\alpha'})$ of $\lambda_1^n$ depending on the strength 
of $\theta$. 
 We consider $\al'>\al$ and $i\in\{2, \ldots, p\}$ and define
\bes\label{18610.19h} \La_n:= \lf]\lambda_{i-1}^n+ {n^{-1+\alpha'}},\lambda_{i}^n-
 {n^{-1+\alpha'}}\ri[.\ees
Note first that if $\La_n$ is empty, then 
the eigenvalue of $\wtX$ which lies between $\la_{i-1}^n$ and $\la_i^n$ is within $n^{-1+\al'}$ to both  $\la_{i-1}^n$ and $\la_i^n$, so we have nothing to prove.
 Now, we want to prove that
 $f_n$ does not vanish on  $\La_n$
 and   that according to the sign of $\frac{1}{\theta}-
\frac{1}{\underline{\theta}}$, it  vanishes on 
one side or the other of $\La_n$ in $]\lambda_{i-1}^n,\lambda_{i}^n
 [$. This will prove (i) and (ii) of the theorem. Part (iii) can be proved in the same way, proving that with overwhelming probability, 
$f_n$ does not vanish in $\lf]\lambda_{n-i-1}^n+ {n^{-1+\alpha'}},\lambda_{n-i}^n-
 {n^{-1+\alpha'}}\ri[ $.

The proof of this fact will follow the same lines as the proof of Lemma \ref{8210.00h29} and we recall 
 that 
$P$ was defined above as  the orthogonal
projection onto the linear span of the eigenvectors of $X_n$ 
corresponding to the eigenvalues $\lambda_{m_n+1}^n,\ldots, \lambda_{n}^n$.
Then, exactly as for \eqref{eqlemma4.8}, we can show that for all $\delta>0$,
$$\sup_{z\in [\la_1^n, \la_p^n]}\lf|\frac{1}{n}\langle g, P(z-X_n)^{-1}P g\rangle -\frac{1}{\underline\theta}\ri|\le\delta$$
with overwhelming probability. Moreover,    for any $z\in \La_n$, for  any $j=1, \ldots,  m_n$, we have
$$| z-\la_j^n|\ge \min \{z-\la_{i-1}^n, \la_i^n-z\}\ge {n^{-1+\al'}},$$ so that  $$\sup_{z\in \La_n}\lf|\frac{1}{n}\langle g,(1- P)(z-X_n)^{-1}(1- P) g\rangle\ri|\le
n^{-\al'}\langle g, (1- P) g\rangle.$$ By Proposition \ref{hanson}, we deduce that 
  for any $\e >0$,  $$\sup_{z\in \La_n}\lf|\frac{1}{n}\langle g,(1- P)(z-X_n)^{-1}(1- P) g\rangle\ri| \le 
 n^\e
n^{-\al'}m_n$$
 with overwhelming probability. We  choose $\e$ in such a way that  the latter right
hand side goes to zero.
Therefore, we know that uniformly on 
$\La_n$, 
$$f_n(z)= \frac{1}{\underline\theta}-\frac{1}{\theta} +o(1)$$
 with overwhelming probability. Since for all $n$, $f_n$
is decreasing, 
 going to $+\infty$ (resp.  $-\infty$) as $z$ goes to any $\la_{i-1}^n$ on the right (resp.  $\la_i^n$ on the left), it follows that according to the sign of
 $\frac{1}{\underline\theta}-\frac{1}{\theta}$, the zero of $f_n$ in $]\la_{i-1}^n, \lambda_{i}^n [$ is either in $]\la_{i-1}^n, \lambda_{i-1}^n+ {n^{-1+\alpha'}} [$ or in 
 $]\la_{i}^n-{n^{-1+\al'}}, \lambda_{i}^n  [$.\qed

  \noindent{\it Proof of Theorem \ref{corstick}.} For each $\ell=0, \ldots, r$, let us define $$\wtX^{(\ell)}:=  X_n+\sum_{i=1}^\ell \tta_iu^n_iu_i^{n^*}
 $$ and denote its eigenvalues by ${\wtl_1}^{(\ell)}\le \cdots\cdots\le \wtl_n^{(\ell)}$. We also define 
 \beq p\ul_-&:=&\sharp\{i=1, \ldots, \ell\ste \rho_{\tta_i}<a\},\\
 p\ul_+&:=&\sharp\{i=1, \ldots, \ell\ste \rho_{\tta_i}>b\}.
\eeq 
$p\ul_-$ and $p_+\ul$ are respectively the numbers of eigenvalues of $\wtX\ul$ with limit $<a$ and $>b$. 
We also set $$f_n\ul(z):=\lan u^n_\ell, (z-\wtX^{(\ell-1)})^{-1}u^n_\ell\ran- \ff{\tta_\ell}.$$ Of course, as before, 
the eigenvalues of $\wtX\ul$ are the zeros of $f_n\ul$. 

Let us also choose $\zeta_a<a$ and $\zeta_b>b$ \st $$\zeta_a>\max\{\rho_{\tta_i}\ste \rho_{\tta_i}<a\}\quad\textrm{ and }\quad \zeta_b<\min\{\rho_{\tta_i}\ste 
\rho_{\tta_i}>b\}.$$

First, as in the proof of Theorem \ref{exact1d}, up to small perturbations, one can suppose that for all $\ell=0,\ldots, r$, the eigenvalues of $\wtX\ul$ 
are pairwise distinct and for all $\ell=1,\ldots, r$, the eigenvalues of $\wtX\ul$ are distinct from those of $\wtX^{(\ell-1)}$. 

 Now, let us state a few facts:
 
(a)   For all $\ell$, there is a constant $M$ \st  the extreme eigenvalues of $\wtX^{(\ell)}$ are in $[-M,M]$ with overwhelming probability 
(this follows from Theorem \ref{241109.17h50}).

(b) Moreover, for each $\ell$, for each $i<k$, by Weyl's interlacing inequalities, $$0\le \ff{\wtl\ul_{i+1}-\wtl\ul_{k-1}}\le\ff{\wtl_i^{(\ell-1)}-\wtl_k^{(\ell-1)}}.$$
which implies, by induction over $\ell$, that $\wtX\ul$ satisfies
the first part  Hypothesis \ref{hypspec} and      \eqref{16511} and \eqref{ecarts_vp_190110} in 
  Hypotheses   \ref{H3a}  $[p-\ell,\alpha,a]$ and $[q-\ell,\alpha,b]$. 

We only consider the i.i.d. model, so each $\wtX^{(\ell)}$ can be deduced from $\wtX^{(\ell-1)}$ by adding an independent rank one perturbation.

In the case where all the $\theta_i$'s are in $[\underline\theta,\overline\theta]$, also the extreme eigenvalues of $\wtX^{(\ell)}$ stick to the bulk and
therefore the full hypothesis \ref{hypspec} holds at each step. In this
case we can simply apply Theorem \ref{exact1d} inductively 
to prove the theorem. The appearance of spikes is in fact not a problem
as Theorem \ref{theostick} insures that for all $\ell$, the
eigenvalues of $\tilde X^{(\ell)}_n$ are close to the eigenvalues
of $X_n$ simultaneously with overwhelming probability, whereas 
Weyl's interlacing properties and as in the previous proof discussions
on the sign of the functions $f_n$ allows to localise in the neighborhood
of which eigenvalues of $X_n$ the eigenvalues of $\tilde X_n^{(\ell)}$
lie. 

 Let us detail a bit this argument.
  By using (a) and (b) above and following the proof of Lemma \ref{8210.00h29}, one can easily prove that for all $\ell=1, \ldots, r$, for any $i=p_-\ul,\ldots, p-\ell$ (resp. $j=p_+\ul,\ldots,  q-\ell$), for any $\delta>0$, for  \be\label{1851116h-1}\Omega_n:=]\max\{\wtl_{i-1}\ul\, ,\, \zeta_a\}+n^{-1+\al'}, \wtl_i\ul-n^{-1+\al'}[ \ee \be\label{1851116h-11}\textrm{(resp. }\Omega_n:=]\wtl_{n-j}\ul+n^{-1+\al'}, \min\{\wtl_{n-j+1}\ul\, ,\, \zeta_b\}-n^{-1+\al'}[\textrm{ )},\ee with overwhelming probability, \be\label{1851116h} \sup_{z\in \Omega_n} \lf|\lan u^n_\ell, (z-\wtX^{(\ell-1)})^{-1}u^n_\ell\ran-\ff{\udl{\tta}}\ri|\le \delta\ee \be\label{1851116h1}  \textrm{(resp. }\sup_{z\in \Omega_n} \lf|\lan u^n_\ell, (z-\wtX^{(\ell-1)})^{-1}u^n_\ell\ran-\ff{\ovl{\tta}}\ri|\le \delta\textrm{ ).}\ee
 
 Let us now fix $\ell\in \{1, \ldots, r\}$ and compare the eigenvalues of 
$\wtX\ul$ to the ones of $\wtX^{(\ell-1)}$. 

We suppose for example that $\tta_\ell>0$. 

Then by Weyl's inequalities,  we have $$
\wtl_1\ulmu <\wtl_1\ul<\wtl_2\ulmu<\wtl_2\ul<\cdots\cdots<\wtl_{n-1}\ulmu <\wtl_{n-1}\ul<\wtl_n\ulmu<\wtl_n\ul.$$

\begin{itemize}\item
Let us first consider the smallest eigenvalues. Under the overwhelming event  \eqref{1851116h}, $f_n\ul<0$ on any interval $\Omega_n$ 
as defined in \eqref{1851116h-1}. So, since $f_n\ul$ is decreasing and vanishes exactly once on $]\wtl_{i-1}\ulmu, \wtl_i\ulmu[$, its zero 
  $\wtl_{i-1}\ul$ is within $n^{-1+\al'}$ from $\wtl_{i-1}\ulmu$.\\  

\item Let us now consider the largest eigenvalues. Under the overwhelming event  \eqref{1851116h1}, $f_n\ul$ has the same sign as $\ff{\ovl{\tta}}-\ff{\tta_\ell}$
 on any interval $\Omega_n$ as defined in \eqref{1851116h-11}, so $\wtl_{n-j}\ul$ is within $n^{-1+\al'}$ from $\wtl_{n-j}\ulmu$ if $\tta_l<\ovl{\tta}$ and
  $\wtl_{n-j}\ul$ is within $n^{-1+\al'}$ from $\wtl_{n-j+1}\ulmu$ if $\tta_l>\ovl{\tta}$. \end{itemize}

 To conclude, up to $n^{-1+\al'}$ errors,    each perturbation by a positive rank one matrix $\tta_\ell u_\ell^n u_\ell^{n^*}$ does move the smallest eigenvalues
 and translates each largest one to the following eigenvalue if $\tta_l>\ovl{\tta}$ and does not move the largest eigenvalues if $\tta_l<\ovl{\tta}$. Of course,
 the analogue result holds for perturbations by negative rank one matrices. 
 
 The theorem follows.\qed

\section{Application to classical models of matrices}\label{secAppl}

Our goal in this section is to show that if $X_n$ belongs to some classical 
ensembles of matrices, the extreme eigenvalues of perturbations of such matrices
have their asymptotics obeying to Theorems \ref{241109.17h50}, \ref{theogauss} and 
\ref{theostick}. For that, a crucial step will be the following statement.
If $(X_n)$ is a sequence of random matrices, we say that it satisfies an hypothesis $H$
{\it in probability} if the probability that $X_n$ satisfies $H$  converges to one as $n$ goes
to infinity (for example, if $H$ states a convergence to a limit $\ell$, ``$H$ in probability" is the convergence in \pro to $\ell$).
\begin{Th} \label{theoprob}
Let $(X_n)$ be a sequence of random matrices independent of the $u_i^n$'s.
Under Assumption \ref{hyponG},
\begin{enumerate}
 \item If Hypothesis \ref{hypspec} holds in probability, Theorem \ref{241109.17h50} holds.
\item If $\kappa_4(\nu)=0$ and Hypotheses \ref{hypspec} and \ref{H2}  hold in probability,
Theorem \ref{theogauss} holds.
If $\kappa_4(\nu)\neq 0$ and Hypotheses \ref{hypspec} and \ref{H2'}  hold in probability,
Theorem \ref{942010.1} holds.
\item Under Assumption \ref{hypthetamin}, if Hypotheses \ref{hypspec} and \ref{H3a} hold in probability,
Theorem \ref{theostick} holds ``with probability converging to one'' instead of ``with overwhelming probability'';
Theorems \ref{exact1d} and Corollary \ref{corstick} hold. 
\end{enumerate}
\end{Th}

This result follows from the results with deterministic sequences of matrices $X_n.$
Indeed, to prove that a sequence converges to a limit $\ell$ in a metric space, it suffices to prove that any of its subsequences has a subsequence converging to $\ell$. If the convergences of the hypotheses  hold in probability, then from any subsequence, one can extract a subsequence for 
which they hold almost surely. Then up to a conditioning by the $\sigma$-algebra generated by the $X_n$'s, 
the hypotheses of the various theorems hold.\\

The remaining of this section is devoted to showing that such results hold
if $X_n$, independent of $(u_i^n)_{1\le i\le r}$, is a Wigner or
a Wishart matrix or a random matrix which law has density proportional to $e^{-\Tr V}$ for a certain potential $V$.
In each case, we have to check that the hypotheses hold in probability.

\subsection{Wigner matrices}
Let $\mu_1$ be a centred distribution on $\R$  (respectively on $ \C$)  and 
 $\mu_2$ be a centred distribution on $\R$,
 both having  a finite fourth moment (in the case where $\mu_1$ is not supported on the real line,
we assume that the real and imaginary part are independent).
 We define $\sigma^2=\int_{z\in \C} |z|^2\ud\mu_1(z)$. 

Let $(x_{i,j})_{ i,j\ge 1}$ be an infinite Hermitian random matrix
which entries are independent up to the condition $x_{j,i}=\ovl{x_{i,j}}$ \st the $x_{i,i}$'s are distributed according to $\mu_2$ and the $x_{i,j}$'s ($i\neq j$) are distributed according to $\mu_1$. We take $X_n=\ff{\sqrt{n}}\left[{x_{i,j}} \right]_{i,j=1}^n,$
which is said to be a {\it Wigner matrix}. 
For certain results, we will also need an additional hypothesis, which we present here: 
\begin{hyp}\label{hypwig}
The probability measures $\mu_1$ and $\mu_2$  have a {\it  sub-exponential decay}, that is there exists  positive constants $C,C'$ \st
if $X$ is distributed according to  $\mu_1$ or $\mu_2,$ for all $t\ge C'$, 
$$\Pro(|X|\ge t^C)\le    e^{-t}.$$ 
Moreover, $\mu_1$ and $\mu_2$ are symmetric.
\end{hyp}
The following Proposition generalises some results
 of   \cite{sandrinePTRF06,FP07,CDF09,CDF09b}
which study the effect of a finite rank perturbation 
on a non-Gaussian Wigner matrix. In particular, 
it includes the study of the eigenvalues 
which stick to the bulk.

\begin{propo} \label{1622010.11h29} Let $X_n$ be a Wigner matrix. Assume that Assumption \ref{hyponG} holds.
The limits of the extreme eigenvalues of $\wtX$ are given by Theorem \ref{241109.17h50}
and the fluctuations of the ones 
which limits are out of $[-2\sigma, 2\sigma]$ 
are given by Theorem \ref{theogauss}, where the parameters $a,b, \rho_{\tta}, c_\al$ are given by the following formulas :
$b=-a=2\sigma$, $$\rho_\tta :=\begin{cases}\tta+\f{\sigma^2}{\tta}&\textrm{if $|\tta|>\sigma$,}\\
2\sigma&\textrm{if  $0<\tta\le\sigma$,}\\
-2\sigma&\textrm{if   $-\sigma\le \tta<0$,}\end{cases}
$$ and $$c_\al=\begin{cases}\sqrt{\alpha^2- \sigma^2}&\textrm{in the i.i.d. model,}\\
\\
\f{\sigma\sqrt{\al^2-\sigma^2}}{\al}&\textrm{in the orthonormalized model.}\end{cases}$$

Assume moreover that, for all $i,$ $\theta_i\not\in
 \{-\sigma,\sigma\}$ and Hypothesis \ref{hypwig} holds.
If the
perturbation has rank one, we have the following precise description of the fluctuations of the sticking eigenvalues :
\begin{itemize}
 \item If  $\theta>\sigma$ (resp.  $\theta<-\sigma$),
for all $p \ge 2$,  $n^{2/3}(\wtl_{n-p+1}^n- 2\sigma)$ (resp. $n^{2/3}(\wtl_p^n+2\sigma)$) 
converges in law to the $p-1$th Tracy Widom law.
\item If $0\le \theta < \sigma$  (resp. $-\sigma < \theta \le 0$), for all $p \ge 1,$
$n^{2/3}(\wtl_{n-p+1}^n- 2\sigma)$ (resp. $n^{2/3}(\wtl_p^n+2\sigma)$) 
converges in law to the $p$th Tracy Widom law.
\end{itemize}
If the perturbation is rank more than one and Assumption \ref{hypthetamin} holds, the extreme eigenvalues of $\wtX$ are at distance 
less than $n^{-1+\e}$ for any $\e>0$ to the extreme eigenvalues of $X_n,$ which have Tracy-Widom fluctuations. We can localize exactly near which eigenvalue of $X_n$ they lie by using Theorem \ref{corstick} in the i.i.d model.
\end{propo}

\begin{rmq}
 All the Tracy-Widom laws involved in the statement of the proposition above,
are the ones corresponding respectively to the GOE if $\mu_1$ is supported on $\R$ and 
to the GUE if $\mu_1$ is supported on $\C.$
\end{rmq}

According to Theorem \ref{theoprob}, it suffices to
 verify that the  hypotheses 
hold in probability for $(X_n)_{n\ge 1}$. We study separately 
the eigenvalues which stick to the bulk and those which deviate
from the bulk.

$\bullet${\it Deviating eigenvalues.}

If $X_n$ is a Wigner matrix (that is, with our terminology, with entries having a finite fourth moment), the fact that $X_n$
satisfies 
Hypothesis \ref{hypspec} in probability is a well known result (see for example \cite[Th. 5.2]{bai-silver-book}) for $\mu_X$ the semicircle law  
 with support $[-2\sigma, 2\sigma]$. The formulas for $\rho_\tta$ and $c_\al$ can be checked with the well known formula \cite[Sect. 2.4]{alice-greg-ofer}: \be\label{eq-trSTsc}\forall z\in\R\bck[-2\sigma, 2\sigma], \qquad G_{\mu_X}(z)=\f{z-\sgn(z)\sqrt{z^2-4\sigma^2}}{2\sigma^2}.\ee

Moreover,  \cite[Th. 1.1]{baiyao2005}
shows that  $\tr(f(X_n))-n\int f(x)d\sigma(x)$ converges in law to a Gaussian distribution for any function $f$ which is analytic in a neighborhood of $[-2\sigma, 2\sigma]$. For any fixed $z\notin [-2\sigma, 2\sigma]$, applied for $f(t)=\ff{z-t}$, we get that $n(G_{\mu_n}(z)-G_{\mu_X}(z))$ converges in law to a Gaussian distribution, hence $\sqrt{n}(G_{\mu_n}(z)-G_{\mu_X}(z))$ converges in \pro to zero, so that
Hypothesis \ref{H2} holds in probability.
 
$\bullet${\it Sticking Eigenvalues.} 

We now assume moreover that the laws of the entries satisfy Hypothesis \ref{hypwig}. In order to lighten the notation, we shall now suppose that $\si=1$. Let us first recall that by 
\cite{Sosh,ruz}, the extreme eigenvalues of the
non-perturbed matrix $X_n$, once re-centred and renormalised by $n^{2/3}$,
converge to the Tracy-Widom law (which depends on whether the entries are complex or real).
We need to verify that Hypothesis
\ref{H3a}[p,$\alpha$]  for
any finite $p$ and an $\alpha<1/3$ 
is fulfilled in probability. By \cite{Sosh},
the spacing between the two smallest eigenvalues of $X_n$ is of
order greater than $n^{-\gamma}$ for
$\gamma>2/3$ 
with probability
going to one and therefore, by the inequality 
$$ \sum_{i=m_n+1}^{n} \ff{(\la_p^{n}-\la_{i}^n)^{k}} \le
(\la_{p+1}^n-\la_p^n)^{1-k}\times\sum_{i=m_n+1}^{n} \ff{\la_i^{n}-\la_{p}^n},\qquad\textrm{($k=2$ or $4$),}$$
it is sufficient to prove the first point of Hypothesis \ref{H3a}[p,$\alpha$]. We shall prove it 
by replacing first the smallest eigenvalue by the edge 
$-2$ thanks to  a lemma that Benjamin Schlein \cite{Benj} kindly communicated 
to us.  We will
then prove that the sum of the inverse of the distance
of the eigenvalues to the edge indeed converges
to the announced limit, thanks to both Soshnikov paper \cite{Sosh} 
(for sub-Gaussian tails) or \cite{ruz} (for finite moments), 
and Tao and Vu article \cite{TV}. 
\begin{lem}[B. Schlein]\label{lemschlein}
Suppose  the entries of $X_n$  have a   uniform sub-exponential tail. 
Then for all $\delta>0$,  for all integer number $p$, 
$$\lim_{n\to\infty}\Pro\lf(\lf|\ff{n}\sum_{j=p+1}^n\ff{\la^n_j-\la_p^n}-\ff{n}\sum_{j=p+1}^n\ff{\la^n_j+2}\ri|\ge \delta\ri)=0.$$
\end{lem}

\begin{pr}
We write $$\ff{n}
\sum_{j=p+1}^n\ff{\la^n_j-\la_p^n}-\ff{n}\sum_{j=p+1}^n\ff{\la^n_j+2}=\f{\la_p^n+2}{n}\sum_{j=p+1}^n
\ff{(\la^n_j-\la_p^n)(\la^n_j+2)}.$$
Hence for any $K_1>0$, 
\bes\hskip-7cm\Pro\lf(\lf|\ff{n}\sum_{j=p+1}^n\ff{\la^n_j-\la_p^n}-
\ff{n}\sum_{j=p+1}^n\ff{\la^n_j+2}\ri|\ge \delta\ri)\ees \beqy\nonumber
&\le& \Pro(|\la_p^n+2|\ge K_1n^{-2/3})\\
\label{ben.1}&&+\Pro\lf( \f{K_1}{n^{5/3}}\sum_{j=p+1}^n\ff{|(\la^n_j-\la_p^n)(\la^n_j+2)|}\ge \delta\textrm{ and } |\la_p^n+2|< K_1n^{-2/3}\ri).\eeqy
Now, for any $K_2>K_1$, on the event  $\{|\la_p^n+2|< K_1n^{-2/3}\}$,
  for any $\kappa>0$, we have
 \beqy\nonumber \f{K_1}{n^{5/3}}
\sum_{j=p+1}^n\ff{|(\la^n_j-\la_p^n)(\la^n_j+2)|}&\le &  \f{K_1}{n^{5/3}}\sum_{\ell=0}^{+\infty}
\f{\mc{N}_n[2K_2n^{-2/3}+\ell n^{-\kappa},2K_2n^{-2/3}+(\ell+1) n^{-\kappa}
]}{(K_2n^{-2/3}+\ell n^{-\kappa})^2}\\
\label{ben.2}&&+ \f{K_1}{n^{5/3}}\sum_{j=p+1}^n
\f{\one_{\la_j+2\le 2K_2n^{-2/3}}}{|(\la^n_j-\la_p^n)(\la^n_j+2)|},
\eeqy
where $\mc{N}_n[a,b]:=\sharp\{i\ste -2+a\le \la_i^n\le -2+ b\}$.
 Note that, from the upper bound on the density of eigenvalues in microscopic intervals, due to \cite[Theorem 4.6]{ESY08},
we know that for any $\kappa<1$,
there is a  constant $M$ independent of $n$ 
so that for all $\ell\ge 1$
  \be\label{density-eig-microscopic-intervals}\E(\mc{N}_n[2K_2n^{-2/3}+\ell n^{-\kappa},2K_2n^{-2/3}+(\ell +1)n^{-\kappa}])\le M n^{1-\kappa}.\ee 
Let us fix $\kappa \in (\frac{2}{3},1)$. 
It follows that the first term of the r.h.s. of \eqref{ben.2} can be estimated by 
\bes\hskip-2cm\Pro\lf( \f{K_1}{n^{5/3}}\sum_{\ell=0}^{+\infty}
\f{\mc{N}_n[2K_2n^{-2/3}+\ell n^{-\kappa},2K_2n^{-2/3}+(\ell+1)n^{-\kappa}
]}{(K_2n^{-2/3}+\ell n^{-\kappa})^2}\ge \f{\delta}{2}\ri)\ees
\beqy\nonumber &\le& \f{2K_1}{\delta n^{5/3}}\sum_{\ell=0}^{+\infty}
\f{\E(\mc{N}_n[2K_2n^{-2/3}+\ell n^{-\kappa},2K_2n^{-2/3}+(\ell
+1)n^{-\kappa}])}{(K_2n^{-2/3}+\ell n^{-\kappa})^2}\\
\nonumber&\le& \f{2MK_1}{\delta n^{2/3}}\frac{1}{n^{\kappa}}
\sum_{\ell=0}^{+\infty}
\f{1}{(K_2n^{-2/3}+\ell n^{-\kappa})^2}
\\
\nonumber&\le &  \f{2MK_1}{\delta n^{2/3}}\frac{1}{n^{\kappa}(K_2n^{-2/3})^2}+
\f{2MK_1}{\delta  n^{\frac{2}{3} }}\int_{0}^{+\infty}\f{\ud t}{(t+K_2n^{-\frac 2 3})^2}\\
\label{ben.3}&\le &\f{2MK_1}{\delta K_2^2n^{\kappa-2/3}}+ \f{2MK_1}{\delta K_2  }.
\eeqy
Let us now estimate   the  second term of the r.h.s. of \eqref{ben.2}.
For any positive integer $K_3$, we have 
\bes\hskip-7cm\Pro\lf( \f{K_1}{n^{5/3}}\sum_{j=p+1}^n
\f{\one_{|\la_j^n+2|\le 2K_2n^{-2/3}}}{|(\la^n_j-\la_p^n)(\la^n_j+2)|}\ge \f{\delta}{2}\ri)\ees
\beqy\nonumber &\le& \Pro\lf(\mc{N}_n(-\infty, 2K_2n^{-2/3}] \ge K_3\ri)+\Pro\lf(\f{K_1K_3}{n^{5/3}} \ff{\min_{p+1\le j\le K_3}|(\la^n_j-\la_p^n)(\la^n_j+2)|}\ge \f{\delta}{2}\ri)\\
\nonumber&\le & \Pro\lf(\la_{K_3}^n\le -2+2K_2n^{-2/3} \ri)+
\Pro\lf( \min_{p\le j\le K_3}| \la^n_j+2|\le \f{\sqrt{2K_1K_3}n^{-5/6} }{ \sqrt{\delta} }\ri)\\ \label{ben.4}&&+\Pro\lf(|\la^n_p-\la_{p+1}^n|\le \f{\sqrt{2K_1K_3}n^{-5/6}}{\sqrt{\delta}} \ri)
\eeqy
From \eqref{ben.1}, \eqref{ben.2}, \eqref{ben.3} and \eqref{ben.4}, we conclude that  \bes\hskip-7cm\Pro\lf(\lf|\ff{n}\sum_{j=p+1}^n\ff{\la^n_j-\la_1^n}-\ff{n}\sum_{j=p+1}^n\ff{\la^n_j+2}\ri|\ge \delta\ri)\ees \beq 
&\le& \Pro(|\la_1^n+2|\ge K_1n^{-2/3})+\f{2MK_1}{\delta K_2  }
 + \Pro\lf(\la_{K_3}\le -2+2K_2n^{-2/3} \ri)\\ &&+\Pro\lf( \min_{1\le j\le K_3}| \la^n_j+2|\le \f{\sqrt{2K_1K_3}n^{-5/6} }{ \sqrt{\delta} }\ri) +\f{2MK_1}{\delta K_2  }+\Pro\lf(|\la^n_2-\la_1^n|\le \f{\sqrt{2K_1K_3}n^{-5/6}}{\sqrt{\delta}} \ri)\eeq
for arbitrary $0<K_1<K_3$ and $K_3\ge 1$. Taking the limit $n\to \infty$, the last two terms disappear, because by \cite[Th. 1.16]{TV}, the distribution of the smallest $K_3$ eigenvalues lives on scales of order $n^{-2/3}\gg n^{-5/6}$. Therefore, 
\begin{multline*}
\lim_{n\to\infty}\Pro\lf(\lf|\ff{n}\sum_{j=2}^n\ff{\la^n_j-\la_1^n}-\ff{n}\sum_{j=2}^n\ff{\la^n_j+2}\ri|\ge \delta\ri)\\ 
\le \lim_{n\to\infty}\Pro(|\la_1^n+2|\ge K_1n^{-2/3})+\f{2MK_1}{\delta K_2  }
 +\lim_{n\to\infty} \Pro\lf(\la_{K_3}\le -2+2K_2n^{-2/3} \ri),
\end{multline*}
still for any $0<K_1<K_3$ and $K_3\ge 1$.
 Now, note that for $K_1$ large enough, the first term can be made as small as we want.  Then, keeping $K_1$ fixed, $K_2$ can be chosen in such a way to make the second term as small as we want too. At last, keeping  $K_2$ fixed, one can choose $K_3$ large enough to make the third term as small as we want
(as can be computed 
since the limit is given by the $K_3$ correlation function of
the Airy kernel).   \end{pr}

To complete the proof of Hypothesis \ref{H3a},
we therefore need to show that

\begin{lem} Assume that the entries of $X_n$ satisfy Hypothesis \ref{hypwig}. Then,
for any $\delta>0$, any finite integer number $p$,
$$\lim_{n\ra\infty}
\Pro\left( \left|\frac{1}{n}\sum_{j=p+1}^n \frac{1}{\lambda_j^n+2}-1\right|>\delta\right)=0$$
\end{lem}
{\it Proof.}
Notice that by \cite{Sosh,ruz} we know that
the $p$ smallest eigenvalues of
$X_n$ converge in law towards the Tracy-Widom law, 
so that  
$$\lim_{\e\downarrow 0} \lim_{n\ra\infty}\Pro\left(\min_{1\le j\le p}|\lambda^n_j+2|<\e n^{-2/3}\right)=0.$$
Thus, for any finite $p$, with large probability,
$$\frac{1}{n}\sum_{j=2}^p \frac{1}{|\lambda^n_j+2|}\le p\e^{-1} n^{-\frac{1}{3}}$$
and therefore it
is enough to prove the lemma for any particular $p$.
As in the previous proof, we choose $p$ large enough so that
$\lambda_p^n\ge -2+n^{-\frac{2}{3}}$
with probability greater than $1-\delta(p)$
with $\delta(p)$ going to zero as $p$ goes to infinity.
 We shall prove that with high probability
\begin{equation}\label{cvb}
\lim_{\gamma\downarrow 0}\lim_{n\ra\infty}\frac{1}{n}
\sum_{j=p}^{[\gamma n]} \frac{1}{\lambda_j^n+2}\le 0.\end{equation}
This is enough to prove the statement as for any $\gamma>0$,
 $2+\lambda^n_{[n\gamma]}$ converges to $\delta(\gamma)>0$
so that $\mu_{X}([{\delta(\gamma)},2])=1-\gamma$, see \cite[Theorem 1.3]{TVLocalization},

$$\lim_{n\ra\infty}\frac{1}{n}\sum_{i=[n\gamma]}^n\frac{1}{\lambda_i^n+2}
=\int_{\delta(\gamma)}^2\frac{1}{2+x}d\mu_X(x),$$
which converges as $\gamma$ goes to zero to $\int (2+x)^{-1} d\mu_X(x)=1$ (by e.g. \eqref{eq-trSTsc}).
To prove \eqref{cvb}, we choose $\rho \in (2/3, \sqrt{2/3})$ and 
 write, on the event $\lambda_j^n+2\ge \lambda_p^n+2\ge n^{-\frac{2}{3}}\ge
n^{-\rho}$ for $j\ge p$, 
$$\frac{1}{n}\sum_{j=p}^{[\gamma n]} \frac{1}{\lambda_j^n+2} \le 
\sum_{1\le k\le K}  n^{ \rho^k-1}
\mc{N}_n[n^{-\rho^{k}}, n^{-\rho^{k+1}}] +\sum_{j=2}^{[\gamma n]} \frac{1_{\lambda^n_j\ge -2+n^{-\rho^{K+1}}}}
{n(\lambda_j^n+2)}=:A_n+B_n.$$
For the first term, we use Sinai-Soshnikov bound, which under the weakest hypothesis are given in \cite[Theorem 2.1]{ruz}. It  implies that with probability
going to one with $M$ going to infinity,
 for $s_n=o(n^{2/3})$ going to infinity, 
$$\sum_{i=1}^n \left(\frac{\lambda_i^n}{2}\right)^{s_n}\le M \frac{n }{s_n^{\frac{3}{2}}}.
$$
This implies, by Tchebychev's inequality and taking $s_n=n^{+\rho^{k+1}}$ that
$$\mc{N}_n[n^{-\rho^{k}}, n^{-\rho^{k+1}}]\le \sharp\left\{i:\left|\frac{\lambda_i}{2}\right|\ge 1-n^{-\rho^{k+1}}\right\}\le 
(1-n^{-\rho^{k+1}})^{-s_n}
\sum_{i=1}^n \left|\frac{\lambda_i^n}{2}\right|^{s_n}\le 
eM  n^{1-\frac{3}{2} \rho^{k+1}}.$$
Consequently we deduce that
$$A_n\le  eM\sum_{1\le k\le K}  n^{ \rho^k}n^{-\frac{3}{2} \rho^{k+1}}\le C n^{-\rho^K(\frac{3}{2}\rho-1)}$$
which goes to zero as $\rho>2/3$. 
For the second term $B_n$,
note that by \cite[Theorem 1.10]{TV}, for any $\e>0$ small enough,
$$\left|\mc{N}_n[ n^{-\e}\ell, n^{-\e}(\ell+1)]
-n\mu_X([-2+ n^{-\e}\ell,-2+ n^{-\e}(\ell+1)])\right|\le n^{1-\delta(\e)}$$
with $\delta(\e)=\frac{2\e-1}{10}$.
Hence, since $\mu_X([-2+ n^{-\e}\ell,-2+ n^{-\e}(\ell+1)])\sim
n^{-\frac{3\e}{2}}\sqrt{\ell}$, we deduce for $\e$ small enough that
for all $\ell\ge 1$,
$$
\mc{N}_n[ n^{-\e}\ell, n^{-\e}(\ell+1)]\le 2n^{1-\frac{3\e}{2}}\sqrt{\ell}.$$
This allows to bound $B_n$ by
$$B_n\le 2\sum_{\ell=1}^{[\gamma n^\e]} \frac{n^\e}{\ell} n^{-\frac{3\e}{2}}\sqrt{\ell}
\le 2\int_0^\gamma\frac{1}{\sqrt{x}} dx= 2\sqrt{\gamma}$$
which goes to zero as $n$ goes to infinity and then $\gamma$ goes to zero. \qed

\subsection{Coulomb Gases}
We can also consider 
random matrices $X_n$  which law is invariant under the action
of the unitary or the orthogonal
group and with eigenvalues with law
given by
\begin{equation}\label{cgas}dP_n(\lambda_1,\ldots,\lambda_n)=\frac{1}{Z_n} |\Delta(\lambda)|^\beta
e^{-n \beta \sum_{i=1}^n V(\lambda_i)} \prod_{i=1}^n d\lambda_i\end{equation}
with a   polynomial  function $V$ of even degree and positive leading coefficient and $\beta=1,2$ or $4$.  
We assume moreover that $V$ is such that the limiting spectral measure $\mu_V$ of $(X_n)$
is connected and compact and that its smallest and largest eigenvalues converge to the boundaries of the support. 
This set of hypotheses is often referred to as the ``one-cut assumption''. It holds in particular if $V$
is strictly convex and this includes the classical
Gaussian ensembles GOE and GUE (with $V(x)= x^2/4$
and $\beta=1,2$).
\begin{propo} \label{1622010.11h29b}
Under the above hypothesis on $V,$ the  extreme eigenvalues 
of $X_n$ converge to the boundary of the support.
The convergence of the extreme  eigenvalues of $\wtX$ is given by Theorem \ref{241109.17h50}.
These eigenvalues  have Gaussian fluctuations as stated in Theorem \ref{theogauss} if they deviate away from the bulk.\\
Suppose moreover that Assumption \ref{hypthetamin} holds.\\
 If the perturbation is
of rank one  and is  strong 
enough so that the largest eigenvalues deviates  from the
bulk, for all $k \ge 2,$ the  rescaled $k$th largest eigenvalue $n^{\frac{2}{3}}
(\widetilde \lambda_{n-k+1}^n-b_V)$    converges weakly towards
the $k-1$-th Tracy Widom  law. If the perturbation is
of rank one  and is  weak 
enough, for all $k \ge 1,$ the  rescaled $k$th largest eigenvalue $n^{\frac{2}{3}}
(\widetilde \lambda_{n-k+1}^n-b_V)$     converges weakly towards
the $k$-th Tracy Widom  law. \\
If the perturbation is of rank more than one, the extreme eigenvalues of $\wtX$
sticking to the bulk are at distance less than $n^{-1+\e}$ for any $\e>0$
from the eigenvalues of $X_n.$  In the i.i.d model, Theorem \ref{corstick}
prescribes exactly in the neighborhood of
which eigenvalues of $X_n$ each of them lie.
\end{propo}

\begin{pr} As explained above, it suffices to
 verify that the  hypotheses 
hold in probability for $(X_n)_{n\ge 1}$.

Note that the convergence of the spectral measure, of
the edges and the fluctuations of the extreme eigenvalues 
were obtained in \cite{shcherbina}.
The fact that $\sqrt{n}(G_{\mu_n}(z)-G_{\operatorname{sc}}(z))$ converges in \pro to zero is a consequence of  \cite{jojo}  so that
Hypothesis \ref{H2} holds. 

We next check  Hypothesis \ref{H3a}{\rm [p,$\alpha$]}  for
the matrix model
$P_n.$ We shall prove it for any $\alpha > 1/3$  and
any integer $p$.
 We first show that
\be\label{1622010.1}\lim_{n\ra\infty}\E\left[\frac{1}{n}\sum_{i\neq p}\frac{1}{\lambda^n_i-\lambda^n_p}\right]
=-V'(a_V)\,.\ee
Indeed, the joint distribution of $(\la_1^n, \ldots, \la_n^n)$ is $$\ff{Z_n^\beta} e^{-n\sum_{i=1}V(\lambda_i)}\prod_{1\le i<j\le n}^n (\lambda_i-\lambda_j)^\beta \one_{\Delta_n} \ud\lambda_1\cdots\ud\la_n,$$
with $\beta=1,2$ or $4$, $Z_n^\beta$ is the normalising constant and
$\Delta_n=\{\lambda_1 < \cdots < \lambda_n\}$.\\
Therefore,
\begin{eqnarray*}
\E\left[{\beta}\sum_{i\neq p}\frac{1}{\lambda^n_i-\lambda^n_{p}
}\right] & = &-
\ff{Z_n^\beta} \int_{\Delta_n}  e^{-n\beta\sum_{i=1}^nV(\lambda_i)}
 \frac{\partial}{\partial \lambda_p}\prod_{1\le i<j\le n}^n (\lambda_i-\lambda_j)^\beta   \ud\lambda_1\cdots\ud\la_n,\\
& = &  \ff{Z_n^\beta} \int_{\Delta_n}   \frac{\partial}{\partial \lambda_p}\left(e^{-n\beta \sum_{i=1}^nV(\lambda_i)} \right)\prod_{1\le i<j\le n}^n (\lambda_i-\lambda_j)^\beta   \ud\lambda_1\cdots\ud\la_n,\\
& = & -n \beta\E\left[V'( \lambda_p^n)\right],
\end{eqnarray*}
by integration by parts.
Equation \eqref{1622010.1} follows, since
$\lambda^n_p$ converges almost surely to $a_V$ (and concentration inequalities
insures $V'( \lambda_p^n)$ is uniformly integrable). 
But, for any $\e>0$,
$$ \frac{1}{n}\sum_{i\neq p}\frac{1}{\lambda^n_i-\lambda^n_{p}
}\ge \frac{1}{n}\sum_{i\neq p}\frac{1}{\e+ \lambda^n_i-\lambda^n_{p}
}$$
with, by convergence of the spectral measure and of $\lambda^n_p$,
the right hand side converging to $-G_{\mu_X}(-a_V -\e)$
which converges as $\e$ decreases to zero to $-G_{\mu_X}(-a_V)=-V'(a_V)$. 
Hence, $ \frac{1}{n}\sum_{i\neq p}\frac{1}{\lambda^n_i-\lambda^n_{p}
}$ is bounded below by $-V'(a_V)$ with large probability for large $n$,
and converges in expectation to  $-V'(a_V)$,
and therefore converges in probability to  $-V'(a_V)$.

Moreover,  by  \cite{shcherbina} (see \cite{TW} in the Gaussian case), 
the joint law of $$\left(n^{2/3}(\lambda_1^n-a_V), n^{2/3}
(\lambda_2^n-a_V),\ldots, n^{2/3}(\lambda_p^n-a_V)\right)$$
converges weakly towards a
probability measure which is absolutely continuous
with respect to Lebesgue measure.  
As a consequence, we also deduce from the first point that $n^{-1}\sum_{i<m_n}
(\lambda_p^n-\lambda_i^n)^{-1}$
  vanishes as $n$ goes
to infinity in probability for $m_n\ll n^{1/3}$
 and therefore \eqref{1622010.1} proves the lacking point of Hypothesis 
\ref{H3a}. 

For the two other points,
observe that  \cite{shcherbina} implies that
  for any $\e>0$,
$
\Pro(|\lambda_2^n-\lambda_{1}^n|\le n^{-\frac{2}{3}-\e})\ninf 0.$
On the event $\{|\lambda_2^n-\lambda_{1}^n|> n^{-\frac{2}{3}-\e}\}$, we have
$|\lambda_i^n-\lambda_{1}^n|> n^{-\frac{2}{3}-\e}$ for all $i\in [2, n-1]$,
so that
\begin{eqnarray*}
\frac{1}{n^2}\sum_{i=2}^{n}\frac{1}{(\lambda_i^n-\lambda_1^n)^2}
\le n^{-\frac{1}{3}+\e}\frac{1}{n}\sum_{i=2}^{n}\frac{1}{\lambda_i^n-\lambda_1^n}\\
\frac{1}{n^4}\sum_{i=2}^{n}\frac{1}{(\lambda_i^n-\lambda_1^n)^4}
\le n^{-1+3\e}\frac{1}{n}\sum_{i=2}^{n}\frac{1}{\lambda_i^n-\lambda_1^n}
\end{eqnarray*}
so that by \eqref{1622010.1} and Markov's inequality, Hypothesis \ref{H3a} holds in probability
for  any $\eta<1/3$, $\eta_4<1$ and $\alpha > 1/3$.
\end{pr}

\subsection{Wishart matrices}\label{wissec}Let $G_n$ be an $n\times m$ real (or complex) matrix with i.i.d. centred entries with law $\mu$ \st $\int z\ud \mu(z)=0$, $ \int |z|^2\ud \mu(z)=1$ and $\int |z|^4 \ud \mu(z)<\infty$.  Let $X_n=G_nG_n^*/m$.

\begin{propo} \label{1622010.11h29-wishart} 
Let $n,m$ tend to infinity in such a way that   $n/m\to c\in(0,1)$.
The limits of the extreme eigenvalues of $\wtX$ are given by Theorem \ref{241109.17h50}
and the fluctuations of those
which limits are out of $[a,b]$ 
are given by Theorem \ref{theogauss}, where the parameters $a,b, \rho_{\tta}, c_\al$ are given by the following formulas: $a=(1-\sqrt{c})^2$, $b=(1+\sqrt{c})^2$
 $$\rho_\tta :=\begin{cases}\tta+\f{\tta}{\tta-c}&\textrm{if $|\tta-c|>\sqrt{c}$,}\\
b&\textrm{if  $|\tta-c|\le\sqrt{c}$ and $\tta>0$,}\\
a&\textrm{if   $|\tta-c|\le\sqrt{c}$ and $\tta<0$,}\end{cases}
$$ and $$c_\al^2=\begin{cases}\al^2 \left(1- \f{c}{(\al-c)^2}\right)&\textrm{in the i.i.d. model,}\\
\\
\f{\al^2c}{(\al-c)^2} \left(1- \f{c}{(\al-c)^2}\right)&\textrm{in the orthonormalised model.}\end{cases}$$

Assume now that the law of the entries satisfy Hypothesis \ref{hypwig}. 
If the perturbation has rank one, we have the following precise description of the fluctuations
of
the extreme  eigenvalues of $\wtX$ :
\begin{itemize}
 \item If  $\theta> c+\sqrt c$ (resp.  $\theta< c-\sqrt c$),
for all $p \ge 2$,  $n^{2/3}(\wtl_{n-p+1}^n- 2\sigma)$ (resp. $n^{2/3}(\wtl_p^n- 2\sigma)$) 
converges in law to the $p-1$th Tracy Widom law.
\item If $0\le \theta < c+\sqrt c$  (resp. $c-\sqrt c < \theta \le 0$), for all $p \ge 1,$
$n^{2/3}(\wtl_{n-p+1}^n- 2\sigma)$ (resp. $n^{2/3}(\wtl_p^n- 2\sigma)$) 
converges in law to the $p$th Tracy Widom law.
\end{itemize}
If the perturbation has rank more than one  and for all $i,$
$\tta_i \notin \{c+\sqrt c, c - \sqrt c\}$, the extreme eigenvalues of $\wtX$ are at distance 
less than $n^{-1+\e}$ for any $\e>0$ to the extreme eigenvalues of $X_n,$ which have Tracy-Widom fluctuations.
\end{propo}
 
 Before getting into the proof, let us make a   remark. The  Proposition above
generalizes some results first appeared in \cite{BBP,{FeralPeche09}}. In these papers,
the authors consider models with multiplicative perturbations (in the sense that the population covariance $\Sigma$
matrix is assumed to be a perturbation of the identity). Here, we consider additive perturbations
but the two models are in fact similar, since  a Wishart matrix can be written as a sum
of rank one matrices $\sum_{i=1}^m \sigma_i Y_i Y_i^*,$ with $\sigma_i$ the eigenvalues of  $\Sigma$
and $Y_i$ $n$-dimensional vectors with i.i.d. entries. So, adding our perturbation
$\sum_{i=1}^r \theta_i U_i U_i^*$ boils down to change $m$ into $m+r$ (the limit of $m/n$ is not changed)  and to extend $\Sigma $ with some new eigenvalues $\tta_1,\ldots, \tta_r$.
 
 \begin{pr} 
Again, it suffices to verify that the  hypotheses 
hold in probability for $(X_n)_{n\ge 1}$. 

 It is known, \cite{Marchenko-Pastur}, that  the spectral measure of $X_n$ converges to the so-called Mar\v cenko-Pastur 
distribution $$\ud\mu_X(x):=\ff{2\pi cx}\sqrt{(b-x)(x-a)}\one_{[a,b]}(x)\ud x,$$where $a=(1-\sqrt{c})^2$ and $b=(1+\sqrt{c})^2$. It is known, \cite[Th. 5.11]{bai-silver-book}, that the extreme eigenvalues converge to the bounds of this support.  
The formula $$G_{\mu_X}(z)=\f{z+c-1-\sgn(z-a)\sqrt{(z-c-1)^2-4c}}{2cz} \qquad\textrm{($z\in \R\bck [a,b]$)}$$ allows to compute $\rho_\tta$ and $c_\al$. 
Moreover, by \cite[Th. 1.1]{bai-silver-2004} or \cite[Th. 9.10]{bai-silver-book}, we also know   that
a central limit theorem holds for the linear statistics
of Wishart matrices, giving Hypothesis \ref{H2} as in the Wigner case.

 For Hypothesis \ref{H3a}, the proof is  similar to the Wigner
case. The convergence to the Tracy-Widom law 
of the non-perturbed matrix is due to S. P\'ech\'e \cite{sandrinePTRF09}
(see  \cite{Nagao} and \cite{Forrester} for the Gaussian case). The approximation of the eigenvalues by the quantiles of the limiting law can be found in 
\cite[Theorem 9.1]{ERYY10} whereas the absolute continuity property needed to
prove Lemma \ref{lemschlein} is derived in \cite[Lemma 8.1]{ERYY10}.
This allows to prove Hypothesis \ref{H3a} in this setting as in the Wigner case,
we omit the details.\end{pr}

\subsection{Non-white ensembles}

In the case of non-white matrices,
we can only study the fluctuations
away from the bulk (since we do not have the appropriate information 
about the top eigenvalues to prove Hypothesis \ref{H3a}).  
We illustrate this generalisation
in a few cases, but it is rather clear that
Theorem \ref{theogauss} applies in a much wider generality.

\subsubsection{Non-white Wishart matrices}

The first statement of Proposition \ref{1622010.11h29-wishart} can be generalised to matrices $X_n$ of the type $X_n=\ff{m}T_n^{1/2}G_nG_n^*T_n^{1/2}$ or $\ff{m} G_nT_nG_n^* $, where $G_n$is  an $n\times m$ real (or complex) matrix with i.i.d. centred entries with law $\mu$ \st $\int z\ud \mu(z)=0$, $ \int |z|^2\ud \mu(z)=1$ and $\int |z|^4 \ud \mu(z)<\infty$
  and $T_n$ is a positive non random Hermitian $n\times n$ matrix with
 bounded operator norm, with a converging empirical spectral law
and with no eigenvalues outside any neighborhood of  the
support of the limiting measure for sufficiently
large $n$. Indeed, in this case, everything, in the proof, stays true (use \cite[Th.1.1]{bai-silver-98}  and \cite[Th. 5.11]{bai-silver-book}). However, when the limiting empirical distribution of $T_n$ is not a Dirac mass,  the computation of the $\rho_\tta$'s and the $c_\al$'s is not easy. 

\subsubsection{Non-white Wigner matrices}
There are less  results in the literature
about the central limit theorem for band matrices 
(with centring with respect to the limit) and
the convergence of the spectrum. We therefore 
concentrate on a special case, namely  a Hermitian matrix $X_n$
with independent Gaussian centred entries so that
 $E[|X_{ij}|^2]= n^{-1}\sigma(i/n,j/n)$
with a stepwise constant function $$\sigma(x,y)=\sum_{i,j=1}^k
1_{\frac{i-1}{k}\le x< \frac{i}{k}\atop \frac{i-1}{k}\le y< {\frac{i}{k}}}
\sigma_{i,j}.$$
In \cite{male},  matrices of the 
form
$S_n=\sum_{j=1}^{k(k+1)} a_j\otimes X_j^{(n)}$
with some independent  matrices $X_j^{(n)}$ from the GUE
and self-adjoint matrices
$a_j$ were studied. Taking $a_j
=(\epsilon_{p,\ell}+\epsilon_{\ell,p})\sigma_{p,\ell}$ or $
i(\epsilon_{p,\ell}-\epsilon_{\ell,p})\sigma_{p,\ell}$ with $\epsilon_{p,\ell}$
the matrix with null entries except at $(p,\ell)$ and $1\le p\le \ell\le k$,
we find that $X_n=S_n$. Then it was proved  \cite[(3.8)]{male}
that there exists
 $\alpha,\epsilon,\gamma>0$
so that  for $z$ with imaginary part greater than  $n^{-\gamma}$ for some $\gamma>0$,

\begin{equation}\label{contqw}
\left|E\left[\frac{1}{n}\tr(z-X_n)^{-1}\right]- G(z)\right|\le (\Im z)^{-\alpha}n^{-1-\e}\end{equation}
which entails the convergence of the spectrum of $X_n$ towards 
the support of the limiting measure \cite[Proposition 11]{male}
with exponential speed by \cite[Proof of Lemma 14]{male}.
Thus $X_n$ satisfies Hypothesis \ref{hypspec}. Hypothesis \ref{H2} can be checked 
by modifying slightly the proof of \eqref{contqw}
which is based on an integration by parts
to be able to take $z$ on the real
line but away from the limiting support. Indeed, as in \cite[Section 3.3]{edouard}, we can add a smooth  cut-off function in the expectation
which vanishes outside of the event $A_n$
that  
$X_n$ has all its eigenvalues within a small neighborhood
of the limiting support. This additional cut-off will
only give a small error in the integration by parts
due to the previous point. Then, \eqref{contqw},
but with an expectation restricted to this event, is proved exactly in the
same way, except that $\Im z$ can be replaced by the distance of $z$
to the neighborhood of the limiting support
where the eigenvalues of $X_n$ lives. 
Finally,  concentration inequalities, in the local version  \cite[Lemma 5.9 and Part II]{aliceStFlour}, insure that
on $A_n$,
 $$\frac{1}{n}\tr(z-X_n)^{-1}-E\left[1_{A_n} \frac{1}{n}\tr(z-X_n)^{-1}\right]$$
is at most of order $n^{-1+\e}$ with overwhelming
probability. This completes the proof of Hypothesis \ref{H2}.

\subsection{Some models for which our hypothesis are not satisfied}\mbox{}\\

We gather hereafter a few remarks about some models 
for which the hypothesis we made on $X_n$ are not satisfied. For sake of simplicity,
we present hereafter only the case of i.i.d. perturbations (1).

\subsubsection{I.i.d. eigenvalues with compact support}
We assume that $X_n$ is diagonal with i.i.d. entries which law $\mu$ is compactly supported. 
As in the core of the paper, we denote by a (resp. b) the left (resp. right) edge of the support 
of $\mu.$ We also denote by $F_\mu$ its cumulative distribution function and assume that there is  $\kappa>0$ \st for all $c>0,$
\be \label{defk}
\lim_{x \ra 0^+} \frac{1-F_\mu(b-cx)}{1-F_\mu(b-x)} = c^\kappa
\ee

In this situation, it is easy to check that Hypothesis \ref{hypspec} holds in probability with $\mu_X=\mu$.
But Hypothesis \ref{H2} is not satisfied. Indeed,  by classical CLT, we have, for $\rho_\alpha \notin [a,b],$
$$ W_\alpha^n=\sqrt n (G_{\mu_n}(\rho_\alpha) - G_\mu(\rho_\alpha))  $$
converges in law, as $n$ goes to infinity to a Gaussian variable $W_\alpha$
 with variance $- G^\prime_{\mu}(\rho_\alpha)- G_\mu(\rho_\alpha)^2.$
Moreover, $$E[W_\alpha W_{\alpha'}]=\int \frac{1}{(\rho_\alpha-\lambda)(\rho_{\alpha'}-\lambda)} d\mu(\lambda)-G_\mu(\rho_\alpha)G_\mu(\rho_{\alpha'}).$$

Nevertheless, Theorem \ref{theogauss} holds for this model. Indeed, the whole proof of this theorem goes through
in this context, except  the proof of Lemma \ref{5110.23h51}, where we have to make the following decomposition $ M_{s,t}^{n}(i,x) =M_{s,t}^{n,1}(i,x) +M_{s,t}^{n,2}(i,x)+
M_{s,t}^{n,3}(i,x) $
with the difference that this time $M_{s,t}^{n,3}$ does not go to zero
but converges towards $W_{\alpha_{i}}$. Hence, the eigenvalues fluctuate according to the distribution of the eigenvalues of 
$(c_jM_j+W_{\al_j}I_{k_j})_{1\le j\le q}$, with $c_j$ and $M_j$ as in the statement of Theorem \ref{theogauss} and $I_{k_j}$ denotes the $k_j\times k_j$ identity  matrix.

Let us now consider the fluctuations near the bulk. We first detail the fluctuations of the extreme eigenvalues of $X_n.$
According to \cite{Gumbel}, the fluctuations of the largest eigenvalues of $X_n$ are determined by the parameter $\kappa$ defined in \eqref{defk},
that is, if $v_n = F_\mu(b-1/n),$ then the law of $\frac{b-\la^n_n}{b-v_n}$ converges weakly to the law with density proportional to $e^{-x^\kappa}$ 
on $\R^+.$ Otherwise stated, the fluctuations of $\la^n_n$ are of order $n^{-1/\kappa}$ with asymptotic distribution the Gumbel distribution of type 2. One can check that if $\kappa \le 1,$ then $\ovl \theta = 0.$\\
One can show that, for any fixed $p,$ for  
Hypothesis  \ref{H3a}$[p,\alpha]$ to hold, we need $\alpha > \frac{1}{\kappa}- \frac{1}{2}$ and we then obtain that the distance of the extreme eigenvalues of the deformed matrix is at distance less that $n^{-1 + \alpha^\prime}$ for any $\alpha^\prime > \alpha.$ 
Therefore if $\kappa > 4/3,$ this theorem allows us to deduce that the  fluctuations of the extreme eigenvalues of the deformed matrix
are the same as those of the non-deformed matrix.



\subsubsection{Coulomb gases with non-convex potentials}
In \cite{pastur}, Pastur showed that for a Coulomb gas law \eqref{cgas} with a potential $V$
so that the equilibrium measure has a disconnected support, the central limit theorem does not hold 
in the sense that the variance may have different limits according to subsequences (see \cite[(3.4)]{pastur}.
Moreover the asymptotics of $\sqrt{n}(\tr(X_n)-\mu(x))$ can be computed sometimes and do
not lead to a Gaussian limit. We might expect then that also $\sqrt{n}(G_{\mu_n}(x)-G_\mu(x))$
converges to a non-Gaussian limit, which would then result with non-Gaussian fluctuations for
the eigenvalues outside of the bulk.

\section{Appendix}
\subsection{Determinant formula}
We here state  formula \eqref{eqint}, which can be deduced from 
the well known formula $\det\left(
\begin{array}{cc}
A&B\cr
C&D\cr
\end{array}\right)= \det(D)\det(A-BD^{-1}C)$.
\begin{lem}\label{detf}
Let $z\in\mathbb C\backslash\{\lambda_1^n,\ldots,\lambda_n^n\}
$ and $\theta_1,\ldots,\theta_r\neq 0$.
Set $D=\diag(\theta_1,\ldots,\theta_r)$ and let $V$ be any $n\times r$ matrix.
Then
$$
\det\left(z-X_n- VDV^*\right)=\det(z-X_n) \det(D)
\det\left( D^{-1} -V^*(z-X_n)^{-1} V\right)$$
\end{lem}

\subsection{Concentration estimates} \label{section.concentration.estimates}
\begin{propo}\label{hanson}
Under Assumption \ref{hyponG}, there exists
a constant $c>0$ so that
 for any matrix $A:=(a_{jk})_{1\le j,k\le n}$
with complex entries, for any $\delta >0$,
for any $g=(g_1,\ldots,g_n)^T$ with i.i.d. entries $(g_i)_{1\le i\le n}$ with law $\nu$, 
$$\mathbb P\left( |\langle g, A g\rangle-\E[\langle g, A g\rangle]
 |>\delta\right)
\le 4 e^{-c\min\{\frac{\delta}{C},\frac{\delta^2}{C^2}\}}$$
if $C^2=\tr (AA^*)$ and if $\tilde g$  is an
 independent copy of
$g$, for any $\delta,\kappa>0$,
$$ {\mathbb P}\left( |\langle g, A \tilde g\rangle
 |>\delta \sqrt{\tr (AA^*)+\kappa\sqrt{\Tr((AA^*)^2)}} \right)
\le 4e^{-c\delta^2}+ 4e^{-c\min\{\kappa,\kappa^2\}}
.$$
\end{propo}

\begin{pr}
The first point is due to  Hanson-Wright Theorem \cite{hanson}, see also 
\cite[Proposition 4.5]{ESY08}. For the second, we use concentration inequalities, see e.g. \cite[Lemma 2.3.3]{alice-greg-ofer}, based on the remark
that for any fixed $\tilde g$, $g\rightarrow \langle g, A \tilde g\rangle$
is Lipschitz with constant $\sqrt{\langle\tilde  g, AA^* \tilde g\rangle}$
and therefore, conditionally
to  $\tilde g$, for any $\delta>0$,
$$\mathbb P\left( |\langle g, A \tilde g\rangle
 |>\delta \sqrt{\langle\tilde  g, AA^* \tilde g\rangle} \right)
\le 4e^{-c\delta^2}$$
On the other hand, the previous estimate shows that
$$\mathbb P\left(|{\langle\tilde  g, AA^* \tilde g\rangle}-\tr(AA^*)|>\kappa\sqrt{\tr(AA^*)^2}\right)\le   4 e^{-c\min\{\kappa,\kappa^2\}}\,.$$
As a consequence, we deduce the second point of the proposition. \end{pr}

Let ${G}^n=\bbm g_1^n\cdots g_r^n\ebm$ be an $n \times r$ matrix  which  columns $g^n_1,\ldots, g^n_r$, are independent copies of an $n\times 1$ matrix with i.i.d. entries   with law $\nu$ and define
$$V_{i,j}^n=\ff{n}\langle g^n_i,g^n_j\rangle,\quad 1\le i,j\le r,$$
and, for $j\le i-1$, if $\det [V^n_{k,l}]_{k,l=1}^{i-1}\neq 0$,
$$W_{i,j}^n=\f{\det [\gamma_{k,l}^{n,j}]_{k,l=1}^{i-1}}{
\det [V^n_{k,l}]_{k,l=1}^{i-1}}, \textrm{ with } 
\gamma_{k,l}^{n,j} = \left\{ \begin{array}{ll}
                   V^n_{k,l},& \textrm{if } l\neq j, \\
- V^n_{k,i},& \textrm{if } l= j .
                  \end{array} \right.
$$
On $\det [V^n_{k,l}]_{k,l=1}^{i-1}=0$, we give to $W_{i,j}^n$ an arbitrary 
value, say one.
Putting $W^n_{ii}=1$ and $W^n_{ij}=0$ for $j\ge i+1$, it is
a standard linear algebra exercise to check that the column vectors
$$v_i^n=\sum_{j=1}^rW^n_{i,j}{g^n_j}=\textrm{$i$th column of $G^n(W^{n})^{T}$} $$
are orthogonal in $\mathbb C^n$. Let us introduce, for $M$ an $r\times r$ matrix,  $\|M\|_\infty=\sup_{1\le i,j\le r}|M_{i,j}|$. We next prove

\begin{propo}\label{convm}
For any $\gamma >0$, there exists
finite positive constants $c,C$ (depending on $r$) so that for $Z^n=V^n$ or $W^n$,
$$\mathbb P\left(\|Z^n-I\|_\infty\ge n^{-\frac{1}{2}}\gamma\right)
\le C \left[e^{-4^{-1}c \gamma^2}+e^{-c\sqrt{n}}\right].$$
Moreover, with $\|v||_2^2=\sum_{i=1}^n |v_i|^2$, for any $\gamma\in (0,\sqrt{n}(2^{-r}-\epsilon)$ for some $\e >0,$
 $$\mathbb P\left(
\max_{1\le i\le r}\left|\frac{1}{n}\|\sum_{j=1}^r Z_{ij}^ng^n_j\|_2^2 -1\right|
\ge  {n^{-\frac{1}{2}}\gamma} 
\right)\le C \left[e^{-4^{-1}c 2^{-r}\gamma^2}+4e^{-c\sqrt{n}}\right].$$
\end{propo}
\begin{pr}
We first consider the case $Z^n=V^n$.
The maximum of $|V_{ij}^n -\delta_{ij}|$  is controlled by the previous proposition 
with $A=n^{-1}I$, and the result follows from $\tr AA^*=n^{-1}$
and $\tr((AA^*)^2)=n^{-3}$, and choosing $\delta=\gamma/\sqrt{2}$, $\kappa=\sqrt n$.
The result for $W^n$ follows as on $\|V^n-I\|_\infty\le \gamma n^{-\frac 1 2}\le 1$
$$| \det [V_{k,l}]_{k,l=1}^{i-1}-1|\le 2^r \gamma n^{-\frac 1 2},$$
whereas 
$$|\det [\gamma_{k,l}^{n,j}]_{k,l=1}^{i-1}|\le  2^r \gamma n^{-\frac 1 2}.$$
For the last  point, we just notice that since $\frac{1}{n}\|\sum_{j=1}^r Z_{i,j}^ng^n_j\|_2^2=(ZVZ^*)_{i,i}$,
we have 
$$\max_{1\le i\le r}\left|\frac{1}{n}\|\sum_{j=1}^r Z_{ij}^ng^n_j\|_2^2 -1\right|
\le C(r)\max_{Z^n=V^n \textrm{ or }W^n}\|Z^n\|_\infty^2 \max_{Z^n=V^n \textrm{ or }W^n}\|Z^n-I\|_\infty$$
for a finite constant $C(r)$ which 
only depends on $r$. Thus the result follows from
the previous point.\end{pr}

\subsection{Central Limit Theorem for quadratic forms}

\begin{Th}\label{3110.23h27}
Let us fix $r\ge 1$ and let, for each $n$, $A^n(s,t)$ ($1\le s,t\le r$)  be a family of $n\times n$ real (resp. complex)   matrices  \st for all $s,t$, $A^n(t,s)=A^n(s,t)^*$ and \st  for all $s,t=1, \ldots, r$, \begin{itemize}\item[$\bullet$]  in the i.i.d. model, \be\label{3110.happe}
\ff{n}\tr[ A^n(s,t)A^n(s,t)^*]\ninf \sigma_{s,t}^2 , \quad\ff{n}\sum_{i=1}^n|A^n(s,s)_{i,i}|^2\ninf \omega_s ,
\ee
\item[$\bullet$] in the orthonormalised model, 
\be\label{3110.happe-ortho}
\ff{n}\tr[|A^n(s,t)-\ff{n}\Tr A^n(s,t)|^2]\ninf \sigma_{s,t}^2 , \quad \ff{n}\sum_{i=1}^n\left|A^n(s,s)_{i,i}-\ff{n}\Tr A^n(s,t)\right|^2\ninf \omega_s .
\ee
\end{itemize}for some finite numbers $\sigma_{s,t}, \omega_s$ (in the case where $\kappa_4(\nu)=0$, the part of the hypothesis related to $\omega_s$ can be removed).
For each $n$, let us define the $r\times r$ random matrix  \bes\label{3110.vosges}G_n:=\lf[\sqrt{n}\left(\langle u^n_s, A^n(s,t) u^n_t\rangle
-\one_{s=t}\frac{1}{n}\tr( A^n(s,s))\right)
 \ri]_{s,t=1}^r.\ees
 Then the distribution of $G_n$ converges weakly to the distribution of a real symmetric (resp. Hermitian)   random matrix $G=[g_{s,t}]_{s,t=1}^r$    \st   the random variables \beq&\{g_{s,t}\ste 1\leq s\le t\le r\}&\\
&\textrm{(resp. $\{g_{s,s}\ste 1\leq s\le r\}\cup\{\Re(g_{s,t})\ste 1\leq s< t\le r\}\cup 
\{\Im(g_{s,t})\ste 1\leq s< t\le r\}$)}&\eeq
are independent 
and  for all $s$, 
$g_{s,s}\sim \mc{N}(0, 2\sigma_{s,s}^2+\kappa_4(\nu)\omega_s)$ (resp. $g_{s,s}\sim \mc{N}(0, \sigma_{s,s}^2+\kappa_4(\nu)\omega_s)$) and for all $s\neq t$, $g_{s,t}\sim \mc{N}(0,\sigma_{s,t}^2)$ (resp.  $\Re(g_{s,t}), \Im(g_{s,t})\sim \mc{N}(0,\sigma_{s,t}^2/2)$).
\end{Th}

\begin{rmq}\label{3110.23h27-rmq}
Note that if the matrices $A^n(s,t)$ depend on a real parameter $x$ in such a way that for all $s,t$, for all $x,x'\in \R$,  \bes
\ff{n}\Tr(A^n(s,t)(x)-A^n(s,t)(x'))^2\ninf 0, 
\ees
then it follows directly from Theorem \ref{3110.23h27} and from a second moment computation that each finite dimensional marginal of the process 
\bes\label{3110.vosges2}\lf[\sqrt{n}\left(\langle u^n_s, A^n(s,t)(x_{s,t}) u^n_t\rangle
-\one_{s=t}\frac{1}{n}\tr( A^n(s,s)(x_{s,s}))\right)
 \ri]_{1\le s,t\le r\,,\; x_{s,t}\in \R\,,\; x_{s,t}=x_{t,s}}\ees converges weakly to the law of a limit process $[g_{s,t}]_{1\le s,t\le r\,,\; x_{s,t}\in \R\,,\; x_{s,t}=x_{t,s}}$ where there is no dependence in the variables $x_{s,t}$ ($1\le s,t\le r$).
\end{rmq}

\begin{pr}$\bullet$ Let us first consider the model where the $(\sqrt{n}u_s^n)_{1\le s\le r}$
are i.i.d. vectors with i.i.d. entries
with law $\nu$ satisfying Assumption \ref{hyponG}. Note that for all $s,t=1, \ldots, r$, by \eqref{3110.happe}, the sequence   $ \ff{n}\sum_{i,j=1}^n A^n(s,t)_{i,j}^2$  is bounded. Hence up to the extraction of a subsequence, one can suppose that it converges to a limit  $ \tau_{s,t}\in \mathbb{C}$. Since the conclusion of the theorem does not depend on the numbers $ \tau_{s,t}$ and the weak convergence is metrisable, one can ignore the fact that these convergences  are only along a subsequence. In the case where $\kappa_4(\nu)=0$, we can in the same way add  the part of the hypothesis related to $\omega_s$.

We have to prove that for any real symmetric (resp. Hermitian)   matrix $B:=[b_{s,t}]_{s,t=1}^r$, the distribution of  $\Tr(BG_n)$ converges weakly to the 
distribution of $\Tr(BG)$. Note that \beq \Tr(BG_n) &=& \ff{\sqrt{n}}(U_n^*C^nU_n-\Tr C^n),
\eeq
where $C^n$ is the $rn\times rn $   matrix and $U_n$ is the $rn\times 1$ random vector  defined by $$C^n=\bbm b_{1,1}A^n(1,1) &\cdots& b_{1,r}A^n(1,r)\\
\vdots && \vdots\\
b_{r,1}A^n(r,1) &\cdots& b_{r,r}A^n(r,r)\ebm,\quad U_n=\sqrt{n}\bbm u^n_1\\ \vdots\\ u_r^n\ebm.$$ 
In the real (resp. complex) case, let us now apply Theorem 7.1  of \cite{bai-yao-TCL} in the case $K=1$.  It follows that the distribution of $$\Tr(BG_n)=\sum_{s=1}^r b_{s,s}G_{n,s,s}+\sum_{1\le s<t\le r}2 \Re(b_{s,t})\Re(G_{n,s,t})+2 \Im(b_{s,t})\Im(G_{n,s,t})$$ converges weakly to a centred real Gaussian law with variance $$\begin{cases} \sum_{s=1}^r b_{s,s}^2 (2\sigma_{s,s}^2+\kappa_4(\nu)\omega_s)+\sum_{1\le s<t\le r}(2b_{s,t})^2\sigma_{s,t}^2&\textrm{in the real case,}\\
\sum_{s=1}^r b_{s,s}^2 (\sigma_{s,s}^2+\kappa_4(\nu)\omega_s)+\sum_{1\le s<t\le r}(2\Re(b_{s,t}))^2\f{\sigma_{s,t}^2}{2}+(2\Im(b_{s,t}))^2\f{\sigma_{s,t}^2}{2}&\textrm{in the complex case.}
\end{cases}$$ It completes the proof in the i.i.d. model.

$\bullet$  In the orthonormalised model,
we can write 
$u^n_s=\frac{1}{\|\sum_{i=1}^s W^n_{si} g_i\|_2}
\sum_{j=1}^s W^n_{sj} g_j$, where the matrix $W^n$ is the one introduced in this section.
It follows that,  with $$B^n(s,t)= A^n(s,t)-\frac{1}{n}\tr(A^n(s,t) ),$$ by orthonormalization of the $u_s^n$'s
\beq&&\sqrt{n}\left( \langle  u^n_s, A^n(s,t) u^n_t\rangle -\frac{\one_{s=t}}{n}\tr(A^n(s,t) )\right)\\
&=&\sqrt{n}\langle  u^n_s, B^n(s,t) u^n_t\rangle\\
&=&\frac{n}{\|\sum_{i=1}^s W^n_{si} g_i\|_2\|\sum_{i=1}^t W^n_{ti} g_i\|_2
} \sum_{j,i=1}^r W^n_{si}\bar W^n_{tj}\frac{1}{\sqrt{n}}
\langle g_i, B^n({s,t}) g_j\rangle. \eeq
But, by the previous result, if $i\neq j$,
$$\frac{1}{\sqrt{n}}
\langle g_i, B({s,t}) g_j\rangle  $$
converges in distribution to a Gaussian law,
whereas  if $i=j$,
$$\frac{1}{\sqrt{n}} 
\langle g_i, B({s,t}) g_i\rangle$$
$$=\frac{1}{\sqrt{n}}\left(
\langle g_i, A({s,t}) g_i\rangle- \E[\langle g_i, A({s,t}) g_i\rangle] \right)+\frac{\tr(A(s,t))}{\sqrt{n}}\left(
\langle g_i,  g_i\rangle- \E[\langle g_i,  g_i\rangle] \right) $$
where both terms converge to
a Gaussian. Thus this term  is also bounded as $n$ goes to infinity.

Hence, by Proposition \ref{convm},
we may and shall replace $W^n$ by the identity (since the error term would be
of order at most $n^{-\frac{1}{2}+\e}$), which yields
$$\sqrt{n} \langle  u^n_s, B^n(s,t) u^n_t\rangle   \approx \sqrt{n}^{-1} \langle g_s, B({s,t}) g_t\rangle $$
so that we are back to the previous setting with
$B$ instead of $A$.
\end{pr}

{\bf Acknowledgments:} We are very grateful to B. Schlein for communicating
us Lemma \ref{lemschlein}. We also thank G. Ben Arous  and J. Baik for fruitful discussions. We also thank the referee, who   pointed some vagueness in the first version of the paper.

\end{document}